\documentclass[a4paper,reqno]{amsart}

\usepackage{amsmath}
\usepackage{amsthm}
\usepackage{amssymb}
\usepackage{amscd} %diagrams

\usepackage{bbm} %bold numbers
\usepackage{mathtools,mathrsfs} %nice calligraphy

\usepackage{verbatim} %DO NOT comment this line (comment environment)

\usepackage{enumitem}

\newtheoremstyle{newremark}
  {5pt}
  {5pt}
  {\rmfamily}
  {}
  {\rmfamily\bf}
  {.}
  {.5em}
  {}

\newtheorem{theorem}{Theorem}
\newtheorem{lemma}[theorem]{Lemma}
\newtheorem{corollary}[theorem]{Corollary}
\newtheorem{proposition}[theorem]{Proposition}

\theoremstyle{newremark}
\newtheorem{remark}[theorem]{Remark}
\newtheorem{definition}[theorem]{Definition}

\newtheorem*{definition*}{Definition} %no numbering for Theorem*
\newtheorem*{notations*}{Notations}

\numberwithin{theorem}{section}
\numberwithin{equation}{section}

\newcommand{\R}{\mathbb{R}} %real numbers

\newcommand{\C}{\mathbb{C}} % complex numbers
\def\XXint#1#2#3{{%
\setbox0=\hbox{$#1{#2#3}{\int}$}
\vcenter{\hbox{$#2#3$}}\kern-.5\wd0}}

\renewcommand{\leq}{\leqslant}
\renewcommand{\geq}{\geqslant}
\renewcommand{\subset}{\subseteq}

\newcommand{\res}{\mathop{\hbox{\vrule height 7pt width .5pt depth 0pt
\vrule height .5pt width 6pt depth 0pt}}\nolimits}

\newcommand{\eps}{\varepsilon}

\newcommand{\e}{{\rm e}}

\newcommand{\de}{{\rm d}}

\begin{document}

%=================
% TITLE AND AUTHOR
%=================

\title[\bf Minimizing $1/2$-harmonic maps into spheres]{Minimizing $1/2$-harmonic maps  into spheres}

\author{Vincent Millot}
%\address{address of the author}
%\email{millot@math.univ-paris-diderot.fr}

\author{Marc Pegon}
\address{Universit\'e Paris-Diderot, Sorbonne Paris-Cit\'e, Sorbonne Universit\'e, CNRS, Laboratoire Jacques-Louis Lions, F-75013 Paris}
\email{millot@math.univ-paris-diderot.fr}
\email{mpegon@math.univ-paris-diderot.fr}

%\date{\today}

%=========
% ABSTRACT
%=========

\begin{abstract}
In this article, we improve the partial regularity theory for minimizing $1/2$-harmonic maps of \cite{MS,Mos}  in the case where the target manifold is the $(m-1)$-dimensional sphere. For $m\geq 3$, we show that minimizing $1/2$-harmonic maps are smooth in dimension 2, and have a singular set of codimension at least~3 in higher dimensions. For $m=2$, we prove that, up to an orthogonal transformation, $x/|x|$ is the unique non trivial $0$-homogeneous minimizing $1/2$-harmonic map from the plane into the circle $\mathbb{S}^1$. As a corollary, each point singularity of a minimizing $1/2$-harmonic maps from a 2d domain into $\mathbb{S}^1$ has  a topological charge equal to $\pm1$. 

\end{abstract}

%\keywords{keywords}

\maketitle

%==================
% TABLE OF CONTENTS
%==================

\tableofcontents

%%%%%%%%%%%%%%%%%%%%%%%%%%%%%%%%%%%%%%%%%%%%%%%%%%%%%%%%%%

\section{Introduction}

In a serie of articles  \cite{DaRi1,DaRi2,Da1,Da2}, F. Da Lio \& T. Rivi\`ere have introduced and studied the  fractional {\sl $1/2$-harmonic maps} from the real line into a manifold. 
Given a compact smooth submanifold  $\mathcal{N}\subset \mathbb{R}^m$ without boundary, $1/2$-harmonic maps into $\mathcal{N}$ are defined as a critical points of the so-called $1/2$-Dirichlet energy under the constraint to be $\mathcal{N}$-valued. 
They naturally appear in several geometric problems such as minimal surfaces with free boundary, see \cite{Da2,DaMaRi,DaPig,Fra,Schi,Struwe}  and Section~\ref{seccirclespheres}. They also come into play in some Ginzburg-Landau models for supraconductivity, see e.g. \cite{BMRS} and  references therein.  The Euler-Lagrange equation satisfied by $1/2$-harmonic maps is in strong analogy with the standard harmonic map system. Instead of the usual Laplace operator, the equation involves the square root Laplacian as defined in Fourier space (i.e., the multiplier operator of symbol $2\pi|\xi|$), and it suffers the same pathologies regarding regularity. A main issue was then to prove the smoothness {\sl a priori} of weak solutions. It has been achieved in  \cite{DaRi1,DaRi2}, thus extending the famous regularity result of F. H\'elein for harmonic maps from surfaces \cite{Hel}. The notion of $1/2$-harmonic maps has  been  extended in \cite{MS,Mos} to higher dimensions, and partial regularity for minimizing or stationary $1/2$-harmonic maps established (again in analogy with minimizing/stationary harmonic maps  \cite{Bet,Ev,SU1}). Before going further, let us now describe in detail the mathematical framework. 
\vskip3pt

 Given a bounded open set $\Omega\subset \R^n$, the $1/2$-Dirichlet energy in $\Omega$ of a measurable map $u:\R^n\to\R^m$ is defined as 
$$\mathcal{E}(u,\Omega):=\frac{\gamma_n}{4}\iint_{(\R^n\times\R^n)\setminus(\Omega^c\times\Omega^c)} \frac{|u(x)-u(y)|^2}{|x-y|^{n+1}}\,\de x\de y\,,$$
where $\Omega^c:=\R^n\setminus\Omega$. 
The normalization constant $\gamma_n:=\pi^{-\frac{n+1}{2}}\Gamma(\frac{n+1}{2})$ is chosen in such a way that 
$$\mathcal{E}(u,\Omega)=\frac{1}{2}\int_{\R^n}\big|(-\Delta)^{\frac{1}{4}}u\big|^2\,\de x \quad\text{for every $u\in\mathscr{D}(\Omega)$}\,.$$
 Following \cite[Section 2]{MS}, we denote by $\widehat H^{1/2}(\Omega;\R^m)$ the Hilbert space made of all  $u\in L^2_{\rm loc}(\R^n;\R^m)$ such that $\mathcal{E}(u,\Omega)<\infty$, and we set 
$$\widehat H^{1/2}(\Omega;\mathcal{N}):=\big\{u\in \widehat H^{1/2}(\Omega;\R^m): u(x)\in\mathcal{N}\text{ for a.e. }x\in\R^n\big\}\,. $$

\begin{definition} 
A map $u\in\widehat H^{1/2}(\Omega;\mathcal{N})$ is said to be a weakly $1/2$-harmonic map in~$\Omega$ with values in $\mathcal{N}$ if 
$$\left[\frac{\de}{\de t}\mathcal{E}\left(\pi_{\mathcal{N}}(u+t\varphi),\Omega\right)\right]_{t=0}=0 \qquad \forall \varphi\in \mathscr{D}(\Omega;\R^m)\,,$$
where $\pi_{\mathcal{N}}$ denotes the nearest point projection on $\mathcal{N}$. 
\end{definition}

According to \cite[Section 4]{MS}, a weakly $1/2$-harmonic map in $\Omega$ satisfies the variational Euler-Lagrange equation
\begin{equation}\label{eqweakintro}
\frac{\gamma_n}{2}\iint_{(\R^n\times\R^n)\setminus(\Omega^c\times\Omega^c)} \frac{\big(u(x)-u(y)\big)\cdot\big(\varphi(x)-\varphi(y)\big)}{|x-y|^{n+1}}\,\de x\de y=0
\end{equation}
for every $\varphi\in H^{1/2}_{00}(\Omega;u^*T\mathcal{N})$. In other words,  \eqref{eqweakintro} holds for every $\varphi\in H^{1/2}_{00}(\Omega;\R^m)$
satisfying $\varphi(x)\in {\rm Tan}(u(x),\mathcal{N})$ for a.e. $x\in\Omega$ (recall that $H^{1/2}_{00}(\Omega)$ is the completion of $\mathscr{D}(\Omega)$ in $H^{1/2}(\R^n)$ for the norm topology). This equation is the weak formulation of the nonlinear system 
\begin{equation}\label{generalhalfharmmapeq}
(-\Delta)^{\frac{1}{2}}u\perp   {\rm Tan}(u,\mathcal{N})\quad\text{in $\Omega$}\,,
\end{equation}
where $(-\Delta)^{\frac{1}{2}}$ is the integro-differential operator given by 
$$(-\Delta)^{\frac{1}{2}}u(x):= {\rm p.v.}\left(\int_{\R^n}\frac{u(x)-u(y)}{|x-y|^{n+1}}\,\de y\right)\,. $$
(The notation ${\rm p.v.}$ means that the integral is taken in the Cauchy principal value sense.) In the case $\mathcal{N}=\mathbb{S}^{m-1}$ (the unit sphere of $\R^m$), the Lagrange multiplier relative to the constraint to be $\mathbb{S}^{m-1}$-valued takes a very simple form, and equation \eqref{generalhalfharmmapeq} rewrites (see \cite[Remark~4.3]{MS})
\begin{equation}\label{halfharmmapeqsphere}
(-\Delta)^{\frac{1}{2}}u(x)= \left(\frac{\gamma_n}{2}\int_{\R^n}\frac{|u(x)-u(y)|^2}{|x-y|^{n+1}}\,\de y\right)u(x)\quad\text{in $\Omega$}\,.
\end{equation}
In this case, it is clear that the right hand side in \eqref{halfharmmapeqsphere} has a priori no better integrability than $L^1(\Omega)$, and thus linear elliptic theory does not apply to determine   the smoothness of solutions.  In \cite{DaRi1,DaRi2} and subsequently in \cite{MazSchi}, the authors have shown that the source term can actually be rewritten in some ``fractional div-curl form". As a consequence, nonlinear compensations appear and the right hand side of \eqref{halfharmmapeqsphere} belongs in fact to the Hardy space. In dimension 1, it leads to continuity and then  full regularity as it happens for harmonic maps in dimension 2 \cite{Hel}.  In higher dimensions, we do not expect any kind of regularity for weakly $1/2$-harmonic maps into a general manifold, again by analogy with weakly harmonic maps in dimensions greater than three \cite{Rivcterex}. However, some partial regularity does hold for minimizing (or at least stationary) $1/2$-harmonic maps. 

\begin{definition} 
A map $u\in\widehat H^{1/2}(\Omega;\mathcal{N})$ is said to be a minimizing $1/2$-harmonic map in $\Omega$ with values in $\mathcal{N}$ if 
$$\mathcal{E}(u,\Omega)\leq \mathcal{E}(v,\Omega) $$
for every competitor $v\in\widehat H^{1/2}(\Omega;\mathcal{N})$ such that  ${\rm spt}(v-u)\subset\Omega$. 
\end{definition}

The result of \cite{MS,Mos} asserts that a minimizing $1/2$-harmonic map $u$ in $\Omega$ belongs to $C^\infty\big(\Omega\setminus {\rm sing}(u)\big)$ where ${\rm sing}(u)$ is the {\sl singular set}  of $u$  in $\Omega$ defined as 
\begin{equation}\label{defsingsetu}
 {\rm sing}(u):=\Omega\setminus\big\{x\in\Omega: \text{$u$ is continuous in a neighborhood of $x$} \big\}\,,
 \end{equation}
which is a relatively closed subset of $\Omega$. Moreover, ${\rm dim}_{\mathcal{H}}\, {\rm sing}(u)\leq n-2$ for $n\geq 3$, and ${\rm sing}(u)$ is locally finite in $\Omega$ for $n=2$ (the notation ${\rm dim}_{\mathcal{H}}$ stands for the Hausdorff dimension), see Corollary \ref{reghalfharm}. 
\vskip3pt

The main purpose of this article is to improve this general regularity result in the case of minimizing $1/2$-harmonic maps into the sphere $\mathbb{S}^{m-1}$. In a first direction, we prove that the size of the singular set can be reduced in case of two or higher dimensional spheres. 

\begin{theorem}\label{mainthm1}
Assume that $m\geq 3$. Let $\Omega\subset  \R^n$ be a smooth bounded open set. If $u\in \widehat  H^{1/2}(\Omega;\mathbb{S}^{m-1})$ is a minimizing $1/2$-harmonic map in $\Omega$, then ${\rm sing}(u)=\emptyset$ for $n\leq 2$,  ${\rm sing}(u)$ is locally finite in $\Omega$ for $n=3$, and  ${\rm dim}_{\mathcal H}\,{\rm sing}(u)\leq n-3$ for $n\geq 4$. 
\end{theorem}

For $m=2$, i.e., in the case of minimizing $1/2$-harmonic maps into $\mathbb{S}^1$, such improved regularity cannot hold for topological reasons, even in dimension 2. To illustrate this fact, let us consider the following variational problem
$$\min\Big\{\mathcal{E}(u,\mathbb{D}): u\in\widehat H^{1/2}(\mathbb{D};\mathbb{S}^1)\,, u(x)=g(x/|x|)\text{ for a.e. }x\in\mathbb{D}^c\Big\}\,, $$
where $\mathbb{D}$ denotes the open unit disc in $\R^2$, and $g:\mathbb{S}^1\to \mathbb{S}^1$ is a smooth given map of non vanishing topological degree. Existence of minimizers easily follows from the direct method of calculus of variations, and any minimizer is obviously a minimizing $1/2$-harmonic map in $\mathbb{D}$. On the other hand, the degree condition on $g$ implies that $g$ does not admit a continuous extension to the whole disc $\overline{\mathbb{D}}$, and thus any minimizer must have at least one singular point. In dimension 2, we already know that the set of singularities is locally finite, and our purpose is to give a description of ``their shape".  This description relies on a blow-up analysis near a singular point (see Section \ref{subsecthm3}), and the study of all possible blow-up limits, usually called {\sl tangent maps}. They turn out to be $0$-homogeneous and minimizing $1/2$-harmonic maps over the whole space (i.e., minimizing  in every ball).  Our next theorem provides the classification of all $0$-homogeneous minimizing $1/2$-harmonic maps from $\R^2$ into $\mathbb{S}^1$.

\begin{theorem}\label{mainthm2}
The map $u_\star:\R^2\to\mathbb{S}^1$ given by $u_\star(x):=\frac{x}{|x|}$ is a minimizing $1/2$-harmonic map in $\R^2$. Moreover, it is the unique non constant $0$-homogeneous minimizing $1/2$-harmonic map up to an orthogonal transformation. In other words,  if $u\in H^{1/2}_{\rm  loc}(\R^2;\mathbb{S}^1)$ is a non constant $0$-homogeneous  minimizing $1/2$-harmonic map in~$\R^2$, then there exists $A\in O(2,\R)$ such that $u(x)=u_\star(Ax)$ for every $x\in\R^2\setminus\{0\}$.
\end{theorem}

As a corollary of Theorem \ref{mainthm2}, we obtain that that a minimizing $1/2$-harmonic map from a two dimensional domain into $\mathbb{S}^1$ must have a degree $\pm 1$ at each  singularity.  The topological degree at a singular point is here defined as the degree of the restriction to any small circle surrounding the point.

\begin{theorem}\label{mainthm3}
Let $\Omega\subset  \R^2$ be a smooth bounded open set. If $u\in \widehat  H^{1/2}(\Omega;\mathbb{S}^{1})$ is a minimizing $1/2$-harmonic map in $\Omega$ and $a\in \Omega\cap{\rm sing}(u)$, then   ${\rm deg}(u,a)\in\{+1,-1\}$. 
\end{theorem}
\vskip5pt

The results and proofs presented in this note represent fractional $H^{1/2}$-counterparts of classical results on minimizing harmonic maps into spheres. First, to prove Theorem~\ref{mainthm1}, we show that a $0$-homogeneous minimizing  $1/2$-harmonic maps from $\R^2$ into $\mathbb{S}^{m-1}$ must be constant if $m\geq 3$. This can be seen as the analogue  of R. Schoen \& K.~Uhlenbeck result \cite[Proposition 1.2]{SU2} about the constancy of $0$-homogeneous minimizing harmonic maps from $\R^3$ into $\mathbb{S}^3$. Their result relies on the fact that a harmonic $2$-sphere into $\mathbb{S}^3$ must be equatorial, a consequence of a theorem of F.J.~Almgren~\cite{Alm}  and E. Calabi \cite{Cal}. Constancy then follows through a second variation argument, destabilizing non constant maps in the orthogonal direction to the image. In our context, any $1/2$-harmonic circle (see Section \ref{sectdefcircles}) turns out to be the boundary of a minimal disc with free boundary in $\mathbb{S}^{m-1}$. Recently, A.M. Fraser \& R. Schoen~\cite{FS} proved that such a minimal disc must be a flat disc through the origin, extending a famous result of J.C.C. Nitsche \cite{Ni} for $m=3$ to arbitrary spheres. As a consequence, any $1/2$-harmonic circle is equatorial (see Corollary \ref{classifthm}), and we use this fact to destabilize non constant $0$-homogeneous $1/2$-harmonic maps from $\R^2$ into $\mathbb{S}^{m-1}$ using again variations in the orthogonal direction to their image (see Proposition \ref{trivtangmap}).  Let us mention that, surprisingly, the same strategy applies to prove smoothness of minimizing ``fractional $s$-harmonic maps" from the line into a sphere for $s\in (0,1/2)$, see \cite{MSY}.  

Concerning Theorem \ref{mainthm2} and Theorem \ref{mainthm3}, we have obtained the $H^{1/2}$-analogue of a classical result of H.~Brezis, J.M. Coron, and E.H. Lieb \cite{BCL} (see also \cite{AlmLi}). In the spirit of \cite{BCL}, the minimality of $x/|x|$ is obtained by means of sharp energy lower bounds, which in turn rely on the distributional Jacobian for $H^{1/2}$-maps into $\mathbb{S}^1$, see \cite{BBM,MilPi,Riv}. To prove the uniqueness part, we use the fact that all $0$-homogeneous $1/2$-harmonic maps in $\R^2$ can be written in terms of finite Blaschke products, which are rational functions of the complex variable. This fact  has been established in \cite{MS} (see also \cite{BMRS,Da1}). Using this representation, we prove rigidity among degree $\pm 1$ maps by domain deformations. Then we exclude maps with higher degree by suitable constructions of competitors in the spirit of \cite[Proof of Theorem 7.4]{BCL}. Compared to~\cite{BCL}, the construction turns out to be more involved as it requires additional steps and the numerical evaluation of certain integrals. Finally, Theorem \ref{mainthm3} is obtained through the aforementioned blow-up analysis near a singularity. More precisely, we prove that homothetic expansions of a minimizing $1/2$-harmonic map near a singular point converge up to subsequences to a non trivial $0$-homogeneous minimizing $1/2$-harmonic map, so that the conclusion follows from Theorem \ref{mainthm2}. Compared to \cite{BCL} again, we do not know if 
a minimizing $1/2$-harmonic map $u$ satisfies $u(x)\sim A(x-a)/|x-a|$ near a singular point $a\in\Omega$ for some $A\in O(2,\R)$, or equivalently if  uniqueness of the blow-up limits  holds. For classical minimizing harmonic maps (into analytic manifolds), uniqueness of blow-ups (i.e., of tangent maps) at isolated singularities has been proved in \cite{Sim,Simbk}. It rests on the so-called \L{}ojasiewicz-Simon inequality, which is not known in our context. 
\vskip3pt

In most of the proofs, we follow the approach of \cite{MS} using of the harmonic extension to the upper half space $\R^{n+1}_+$ given by the convolution with the Poisson Kernel. This allows us to realize the $1/2$-Laplacian as the associated  Dirichlet-to-Neumann map (see Section \ref{sectionprelim}), and then rephrase the $1/2$-harmonic map equation as a harmonic map system with (partially) free boundary condition, see Section \ref{section1}. In particular, we make  use of the existing regularity and compactness results of R.~Hardt  \& F.H.~Lin~\cite{HL}, F.~Duzaar \& K. Steffen \cite{DuS,DuS2}, and F. Duzaar \& J.F.~Grotowski~\cite{DG1,DG2}, see Section~\ref{secharmfreebd}.

\subsection*{Notation}

Throughout the paper, $\R^{n+1}_+$ is the open upper half space $\R^n\times (0,\infty)$, and $\R^n$ can be identified with $\partial  \mathbb{R}^{n+1}_+=\R^n\times\{0\}$. 
More generally, a set $A\subset\mathbb{R}^n$ can be identified with $A\times\{0\}\subset\partial  \mathbb{R}^{n+1}_+$. 
Points in $\mathbb{R}^{n+1}$ are written $\mathbf{x}=(x,x_{n+1})$ with $x\in\mathbb{R}^n$ and $x_{n+1}\in\mathbb{R}$.  
We shall denote by $B_r(\mathbf{x})$ the open ball in $\mathbb{R}^{n+1}$ of radius $r$ centered at $\mathbf{x}=(x,x_{n+1})$, while $D_r(x)$  is the open ball (or disc) in $\R^n$ centered at $x$ (and thus $D_r(x)\times\{0\}= B_r\big((x,0)\big)\cap(\mathbb{R}^{n}\times\{0\})$). If the center is at the origin, we simply write $B_r$ and $D_r$ the corresponding balls. In case $n=2$, we write $\mathbb{D}:=D_1$. 
\vskip3pt

\noindent$\bullet$ For an arbitrary set $G\subset  \mathbb{R}^{n+1}$, we define 
$$G^+:=G\cap \mathbb{R}^{n+1}_+\quad\text{ and }\quad\partial^+ G:=(\partial G)^+=\partial G\cap \mathbb{R}^{n+1}_+\,.$$

\noindent$\bullet$ If $G\subset\R^{n+1}_+$ is a bounded open set, we shall say that $G$ is  {\bf admissible} whenever 
\begin{itemize}
\item[(i)] $\partial G$ is Lipschitz regular;  
\vskip2pt
\item[(ii)] the (relative) open set $\partial^0G\subset\R^n\times\{0\}$ defined by 
$$\partial^0G:=\Big\{\mathbf{x}\in\partial G\cap\partial\R^{n+1}_+ : B^+_{r}(\mathbf{x})\subset G \text{ for some $r>0$}\Big \}$$
is non empty and has a Lipschitz boundary in $\R^n$; 
\vskip2pt

\item[(iii)] $\partial G=\partial^+ G\cup\overline{\partial^0G}\,$.
\end{itemize}
According to this definition, an half ball $B_r^+$ is admissible, and $\partial^0B_r^+=D_r\times\{0\}$.    
\vskip3pt

\noindent$\bullet$ The tangent space to a manifold $\mathcal{N}$ at a point $p\in\mathcal{N}$ is denoted by ${\rm Tan}(p,\mathcal{N})$ (while the tangent bundle of $\mathcal{N}$ is simply denoted by $T\mathcal{N}$). 
\vskip3pt

\noindent$\bullet$ We often identify $\R^2$ with the complex plane $\C$, and if $x=(x_1,x_2)\in \R^2$, the complex variable is written $z:=x_1+i x_2$. Functions taking values into $\R^2$ are also understood as complex valued functions. The product of two such functions are thus understood in the sense of complex multiplication. 
\vskip3pt

Finally, we always denote by $C$ a generic positive constant which may only depend on the dimension $n$, and possibly changing from line to line. If a constant depends on additional given parameters, we shall write those parameters using the subscript notation.

%%%%%%%%%%%%%%%%%%%%%%%%%%%%%%%%%%%%%%%%%%%%%%%%%%%%%%%
%%%%%%%%%%%%%%%%%%%%%%%%%%%%%%%%%%%%%%%%%%%%%%%%%%%%%%%
   								       						%%%%%%%%%%%%%%%%%%%%
\section{Harmonic extension \& the $1/2$-Laplacian} \label{sectionprelim}     %%%%%%%%%%%%%%%%%%%%
								 						%%%%%%%%%%%%%%%%%%%
%%%%%%%%%%%%%%%%%%%%%%%%%%%%%%%%%%%%%%%%%%%%%%%%%%%%%%%
%%%%%%%%%%%%%%%%%%%%%%%%%%%%%%%%%%%%%%%%%%%%%%%%%%%%%%%

\subsection{Harmonic extension} 
For a measurable function $u:\R^n\to \R^m$, we  denote by $u^\e$ its extension to the upper half-space $\R^{n+1}_+$ given by the convolution of $u$ with the Poisson kernel, i.e., 
$$u^{\e}({\bf x}):=\gamma_n\int_{\R^n}\frac{x_{n+1}u(y)}{(|x-y|^2+x^2_{n+1})^{\frac{n+1}{2}}}\,\de y\quad\text{for ${\bf x}=(x,x_{n+1})\in\R^{n+1}_+$}\,. $$
This extension is well defined whenever $u$ belongs to the Lebesgue $L^p$ over $\R^n$ with respect to the finite measure $(1+|x|^2)^{-\frac{n+1}{2}}\,\de x$ for some $1\leq p\leq\infty$. 
In this case, it is well known that $u^{\e}$ provides an harmonic extension of $u$ to $\R^{n+1}_+$. In other words, $u^\e$ solves  
$$ 
\begin{cases}
\Delta u^\e =0 & \text{in $\R^{n+1}_+$}\,,\\
u^\e=u & \text{on $\partial\R^{n+1}_+=\R^n\times\{0\}$}\,.
\end{cases}
$$
Moreover, $u^\e\in L^\infty(\R^{n+1}_+)$ whenever $u\in L^\infty(\R^n)$, and 
\begin{equation}\label{Linftyextineq}
\|u^\e\|_{L^\infty(\R^{n+1}_+)}\leq \|u\|_{L^\infty(\R^n)}\,.
\end{equation}
\vskip3pt

We shall make use of the following  lemma about the harmonic extension. Using the Fourier transform\footnote{Recall that the Fourier transform of the Poisson kernel is given by ${\rm exp}(-2\pi x_{n+1}|\xi|)$.}, its proof is elementary and it is left to the reader.  

\begin{lemma}\label{estiprelim}
If $u\in L^2(\R^n)\cap L^1(\R^n)$, then 
$$\int_{\R^n}|u^\e(x,x_{n+1})|^2\,\de x\leq  \|u\|^2_{L^2(\R^n)}\quad \forall x_{n+1}>0\,,$$
and
$$\int_{\R^n}|u^\e(x,x_{n+1})|^2\,\de x\leq  \frac{C \|u\|^2_{L^1(\R^n)}}{x^n_{n+1}}\quad\forall x_{n+1}> 0\,,$$
for a constant $C$ depending only on $n$. 
\end{lemma}

We complete this subsection recalling the classical identity relating the $H^{1/2}$-seminorm over $\R^n$ with the Dirichlet energy of the harmonic extension:
\begin{multline}\label{iden1/2H1}
\frac{\gamma_n}{4}\iint_{\R^n\times\R^n}\frac{|u(x)-u(y)|^2}{|x-y|^{n+1}}\,\de x\de y= \frac{1}{2}\int_{\R^{n+1}_+}|\nabla u^\e|^2\,\de {\bf x}\\
=\min\left\{ \frac{1}{2}\int_{\R^{n+1}_+}|\nabla v|^2\,\de {\bf x}: v\in \dot H^1(\R^{n+1}_+;\R^d)\,,\; v=u\text{ on }\partial\R^{n+1}_+\right\}
\end{multline} 
for every $u$ in the homogeneous Sobolev space $\dot H^{1/2}(\R^n;\R^m)$.

\subsection{The  $1/2$-Laplacian and the Dirichlet-to-Neumann map} 
Given a smooth bounded open set $\Omega\subset\R^n$, the $1/2$-Laplacian $(-\Delta)^{\frac{1}{2}}:\widehat H^{1/2}(\Omega;\R^m)\to\big(\widehat H^{1/2}(\Omega;\R^m)\big)^\prime $ is defined as the continuous linear operator induced by the quadratic form $\mathcal{E}(\cdot,\Omega)$. For $u\in \widehat H^{1/2}(\Omega;\R^m)$, the action of $(-\Delta)^{\frac{1}{2}}u$ on an  element $\varphi\in \widehat H^{1/2}(\Omega;\R^m)$ is denoted by $\big\langle (-\Delta)^{\frac{1}{2}}u , \varphi\big\rangle_\Omega$, and it is given by 
\begin{equation}\label{weaksqrtlap}
\big\langle (-\Delta)^{\frac{1}{2}}u , \varphi\big\rangle_\Omega=\frac{\gamma_n}{2}\iint_{(\R^n\times\R^n)\setminus(\Omega^c\times\Omega^c)}\frac{(u(x)-u(y))\cdot(\varphi(x)-\varphi(y))}{|x-y|^{n+1}}\,\de x\de y\,.
\end{equation}
Note that, when restricted to $H^{1/2}_{00}(\Omega;\R^m)$, the distribution $(-\Delta)^{\frac{1}{2}}u$ actually belongs to $H^{-1/2}(\Omega;\R^m)$. 

It is well known that the fractional Laplacian $(-\Delta)^{\frac{1}{2}}$ coincides with the {\sl Dirichlet-to-Neumann operator} associated with the harmonic extension to $\R^{n+1}_+$. To be more specific, if $u\in \widehat H^{1/2}(\Omega;\R^m)$, then   $u^\e$ is well defined, and $u^\e\in H^1(G;\R^m)$  for every admissible bounded  open set $G\subset\R^{n+1}_+$ satisfying $\overline{\partial^0G}\subset \Omega\times\{0\}$. Hence, $u^\e$ admits a distributional exterior normal derivative $\partial_\nu u^\e$ on $\Omega\times\{0\}$. By harmonicity of $u^\e$, its action on $\varphi\in\mathscr{D}(\Omega;\R^m)$ can be defined as 
\begin{equation}\label{neumdef}
\big\langle\partial_\nu u^\e,\varphi \big\rangle_\Omega:=\int_{\R^{n+1}_+} \nabla u^\e\cdot\nabla \Phi\,\de{\bf x} \,, 
\end{equation}
where $\Phi$ is any smooth extension of $\varphi$ compactly supported in $\R^{n+1}_+\cup(\Omega\times\{0\})$.  By approximation, the same identity holds for any $\Phi\in H^1(\R^{n+1}_+;\R^m)$ compactly supported in $\R^{n+1}_+\cup(\Omega\times\{0\})$.  In this way, 
the distribution $\partial_\nu u^\e$ appears to  belong to $H^{-1/2}_{00}(\Omega;\R^m)$, and the following identity holds (see \cite[Lemma 2.9]{MS}) 
\begin{equation}\label{idenerumsqrtlap}
\big\langle\partial_\nu u^\e,\varphi \big\rangle_\Omega= \big\langle (-\Delta)^{\frac{1}{2}}u , \varphi\big\rangle_\Omega\qquad \forall \varphi\in H^{1/2}_{00}(\Omega;\R^m)\,.
\end{equation}
All   details can be found in \cite[Section 2]{MS}.

%%%%%%%%%%%%%%%%%%%%%%%%%%%%%%%%%%%%%%%%%%%%%%%%%%%%%%%
%%%%%%%%%%%%%%%%%%%%%%%%%%%%%%%%%%%%%%%%%%%%%%%%%%%%%%%
   								       						%%%%%%%%%%%%%%%%%%%%
\section{$1/2$-harmonic maps vs harmonic maps with free boundary} \label{section1}     %%%%%%%%%%%%
								 						%%%%%%%%%%%%%%%%%%%
%%%%%%%%%%%%%%%%%%%%%%%%%%%%%%%%%%%%%%%%%%%%%%%%%%%%%%%
%%%%%%%%%%%%%%%%%%%%%%%%%%%%%%%%%%%%%%%%%%%%%%%%%%%%%%%

\subsection{Minimizing harmonic maps with free boundary}\label{secharmfreebd} 

For an  admissible bounded open set $G\subset \R^{n+1}_+$, we consider the Dirichlet energy ${\bf E}(\cdot,G)$ defined on $H^1(G;\R^m)$ by  
\begin{equation}\label{defengDirE}
{\bf E}(v,G):=\frac{1}{2}\int_{G}|\nabla v|^2\,\de {\bf x}\,. 
\end{equation}
We also consider a given smooth submanifold $\mathcal{N}\subset \R^m$ that we assume to be compact and without boundary. 

\begin{definition}
Let $G\subset \R^{n+1}_+$ be an admissible bounded  open set, and consider a map $v\in H^1(G;\R^m)$ satisfying $v({\bf x})\in \mathcal{N}$ $\mathcal{H}^n$-a.e. on $\partial^0G$. We say 
that $v$ is a minimizing harmonic map in $G$ with respect to the partially free boundary condition $v(\partial^0G)\subset \mathcal{N}$ if 
$${\bf E}(v,G)\leq {\bf E}(w,G) $$
for every competitor $w\in H^1(G;\R^m)$ satisfying $w({\bf x})\in \mathcal{N}$ for $\mathcal{H}^n$-a.e. ${\bf x}\in\partial^0G$, and such that ${\rm spt}(w-v)\subset G\cup\partial^0G$. {\sl In short, we may say that $v$ is a minimizing harmonic map with free boundary in $G$.}
\end{definition}

Using variations supported in the open set $G$, one obtains that a minimizing harmonic map $v$ with free boundary  is harmonic in $G$, i.e., 
$$\Delta v=0\quad\text{in $G$}\,. $$
In particular, $v\in C^\infty(G)$ by standard elliptic theory. Hence the regularity issue is at the (partially) free boundary $\partial^0 G$. As in \cite{DuS,HL}, one obtains from minimality the boundary condition 
$$\frac{\partial v}{\partial \nu}\perp{\rm Tan}(v,\mathcal{N})  \quad\text{on $\partial^0 G$}\,,$$
which has to be understood in the weak sense, that is
$$\int_G\nabla v\cdot\nabla \zeta\,\de {\bf x}=0 $$
for every $\zeta\in H^1(G;\R^m)$ satisfying $\zeta({\bf x})\in {\rm Tan}(v({\bf x}),\mathcal{N})$ for $\mathcal{H}^n$-a.e. ${\bf x}\in\partial^0G$ and such that ${\rm spt}(\zeta)\subset G\cup\partial^0G$. 
\vskip5pt

Assuming that $v\in L^\infty(G)$, one may apply the (partial) regularity results of \cite{DuS,HL} to derive the following theorem (see \cite[Section 4]{MS} or \cite{Mos}). In its statement, ${\rm sing}(v)$ denotes the so-called {\sl singular set} of $v$ (in $\partial^0G$), i.e., 
$${\rm sing}(v):=\partial^0G\setminus\big\{{\bf x}\in\partial^0G: \text{$v$ is continuous in a neighborhood of ${\bf x}$}\big\}\,,$$
 which turns to be a relatively closed subset of $\partial^0G$. 
  
\begin{theorem}\label{regfreebd}
Let $v\in H^1(G;\R^m)\cap L^\infty(G)$ satisfying $v(\partial^0G)\subset\mathcal{N}$ be a minimizing harmonic map with free boundary in $G$. Then $v\in C^\infty\big((G\cup\partial^0G)\setminus{\rm sing}(v)\big)$,  ${\rm sing}(v)$ is locally finite in $\partial^0G$ for $n=2$, and ${\rm dim}_{\mathcal{H}}\,{\rm sing}(v)\leq n-2$ for $n\geq 3$. 
\end{theorem}

By means of Federer's dimension reduction principle, the size of the singular set can be further reduced according to the existence or non existence of non trivial {\sl tangent maps}. Those maps are defined as all possible blow-up limits of minimizing harmonic maps with free boundary at a point of the free boundary $\partial^0G$, see \cite[Section 3.5]{HL}. In our setting, they  appear to be  $0$-homogeneous maps $v_0\in H^1_{\rm loc}(\overline{\R^{n+1}_+};\R^m)\cap L^\infty(\R^{n+1}_+)$ satisfying  $v_0(\partial \R^{n+1}_+)\subset\mathcal{N}$ which are minimizing harmonic maps with free boundary in $B_R^+$ for every $R>0$. Applying \cite[Theorem~3.6]{HL} (see also \cite[Remark~4.3]{DuS2}), we readily obtain the following result. 

\begin{theorem}\label{dimred}
Let $\ell=\ell(\mathcal{N})$ be the largest integer such that any bounded and $0$-homogeneous minimizing harmonic map with free boundary $v_0:\R^{j+1}_+\to \R^m$ with $v_0(\partial\R^{j+1}_+)\subset\mathcal{N}$ is a constant for each $j=1,\ldots,\ell$. For any minimizing harmonic map $v$ with free boundary as in Theorem~\ref{regfreebd}, we have ${\rm sing}(v)=\emptyset$ if $n\leq \ell$, ${\rm sing}(v)$ is locally finite in~$\partial^0G$ if $n=\ell+1$, and ${\rm dim}_{\mathcal{H}}\,{\rm sing}(v)\leq n-\ell-1$ if $n\geq \ell+2$. 
\end{theorem}

\begin{remark}\label{rembd0hom}
Note that, in applying  \cite{HL}, we use the fact that 
any bounded and $0$-homogeneous minimizing harmonic map with free boundary $v_0$ satisfies the uniform bound 
$$\|v_0\|_{L^\infty(\R^{j+1}_+)}=\|{v_0}_{|\R^j\times\{0\}}\|_{L^\infty(\R^j)}\leq C_{\mathcal{N}}\,,$$
where $C_{\mathcal{N}}$ is (essentially) the width of $\mathcal{N}$ (assuming that $0\in\mathcal{N}$). This estimate follows from the fact $v_0$ is precisely given by the harmonic extension to $\R^{j+1}_+$ of its restriction to $\R^j\times\{0\}$. In other words, if we set 
$u_0:={v_0}_{|\R^j\times\{0\}}$, then $v_0=(u_0)^\e$ (the convolution product of $u_0$ with the $j$-dimensional Poisson kernel).  Indeed, the difference $v_0-(u_0)^\e$ is a bounded harmonic function in $\R^{j+1}_+$. Since it vanishes on $\partial\R^{j+1}_+$,  it has to vanish identically  by the classical Liouville theorem. 
 \end{remark}

We conclude this subsection with an important compactness result for minimizing harmonic maps with free boundary (on which  Theorem \ref{regfreebd} and Theorem \ref{dimred} are based). It corresponds to a weaker version of a more general compactness theorem obtained in \cite[Theorem 2.2]{DG1} (see also \cite[Theorem 2.2]{DG2}). 

\begin{theorem}[\bf compactness]\label{compactfreebd}
Let $(v_k)\subset H^1(G;\R^m)$ be a bounded sequence of minimizing harmonic maps in $G$ with respect to the partially free boundary condition $v_k(\partial^0G)\subset \mathcal{N}$. There exist a (not relabeled) subsequence and $v\in H^1(G;\R^m)$ a minimizing harmonic map  with free boundary in $G$ such that $v_k\to v$ strongly in $H^1_{\rm loc}(G\cup\partial^0G)$. 
\end{theorem}

\subsection{Harmonic extension of minimizing $1/2$-harmonic maps}\label{seccorrespond}

In this subsection, our aim is to prove that minimizing $1/2$-harmonic maps and minimizing harmonic maps with free boundary can be made in one-to-one correspondance by means of the harmonic extension.  
It has been proven in \cite[Proposition 4.9]{MS} that the harmonic extension of a minimizing $1/2$-harmonic map returns a minimizing harmonic map with free boundary  in the upper half space. We shall improve this result showing that a converse statement  holds true. Here again, $\mathcal{N}\subset \R^m$ denotes a given smooth and compact submanifold without boundary. 

\begin{theorem}[\bf minimality transfer]\label{equivmin}
Let $\Omega\subset\R^n$ be a bounded smooth open set. A map $u\in \widehat H^{1/2}(\Omega;\mathcal{N})$ is a minimizing $1/2$-harmonic map in $\Omega$ if and only if its harmonic extension $u^\e$  
is a minimizing harmonic map with free boundary in every admissible bounded  open set $G\subset\R^{n+1}_+$ such that $\overline{\partial^0G}\subset \Omega\times\{0\}$. 
\end{theorem}

\begin{proof}
According to \cite[Corollary 2.10 \& Proposition 4.9]{MS}, 
if $u\in \widehat H^{1/2}(\Omega;\mathcal{N})$ is a minimizing $1/2$-harmonic map in $\Omega$, then $u^\e$ is a minimizing harmonic map with free boundary in every admissible bounded  open set $G\subset\R^{n+1}_+$ such that $\overline{\partial^0G}\subset \Omega\times\{0\}$. It hence remains to prove the converse statement. We thus assume that $u^\e$ is minimizing harmonic map with free boundary in every admissible bounded  open set $G\subset\R^{n+1}_+$ such that $\overline{\partial^0G}\subset \Omega\times\{0\}$. 
\vskip5pt

\noindent{\it Step 1.} We consider an arbitrary competitor $w\in \widehat H^{1/2}(\Omega;\mathcal{N})$, and we assume that $h:=w-u$ is compactly supported in an open set $\Omega^\prime\subset\R^n$ with $\overline{\Omega^\prime}\subset\Omega$. The map $h$ being compactly supported in $\Omega^\prime$, it belongs to $H^{1/2}_{00}(\Omega;\R^m)\cap  L^1(\R^n)$. In view of identity \eqref{iden1/2H1}, its harmonic extension $h^{\rm e}$ belongs to  the homogeneous Sobolev space $\dot H^1(\R^{n+1}_+;\R^m)$. 

We claim that there exists a sequence $(h_k)\subset H^1(\R^{n+1}_+;\R^m)$ such that each $h_k$ is supported in $G_k\cup\partial^0G_k$ for some admissible bounded open set $G_k\subset\R^{n+1}_+$ satisfying $\overline{\partial^0 G_k}\subset \Omega\times\{0\}$, ${h_k}_{|\R^n\times\{0\}}=h$, and 
\begin{equation}\label{approxequmin}
\int_{\R^{n+1}_+}|\nabla h_k|^2\,\de {\bf x}\mathop{\longrightarrow}\limits_{k\to\infty}\int_{\R^{n+1}_+}|\nabla h^{\rm e}|^2\,\de{\bf x}\,.
\end{equation}
Before proving this claim, we complete the proof of the theorem.

By assumption $u^{\rm e}$ is a minimizing harmonic map with free boundary in $G_k$. Since $u^{\rm e}+h_k$ is an admissible competitor for the minimality of $u^{\rm e}$ in $G_k$, we infer that  
\begin{equation}\label{testminhk}
\frac{1}{2}\int_{\R^{n+1}_+}|\nabla h_k|^2\,\de{\bf x}+\int_{\R^{n+1}_+}\nabla h_k\cdot \nabla u^{\rm e}\,\de{\bf x} ={\bf E}(u^{\rm e}+h_k,G_k)- {\bf E}(u^{\rm e},G_k)\geq 0\,. 
\end{equation}
On the other hand, \eqref{neumdef} and \eqref{idenerumsqrtlap} yield 
\begin{equation}\label{ven0953}
\int_{\R^{n+1}_+}\nabla h_k\cdot \nabla u^{\rm e}\,\de {\bf x}=\big\langle (-\Delta)^{\frac12}u,h\big\rangle_{\Omega} \,,
\end{equation}
since $h_k=h$ on $\R^n\times\{0\}$. Letting $k\to\infty$ in \eqref{testminhk}, we deduce from \eqref{approxequmin} and \eqref{ven0953} that 
\begin{equation}\label{vac1449}
\frac{1}{2} \int_{\R^{n+1}_+}|\nabla h^{\rm e}|^2\,\de{\bf x}+\big\langle (-\Delta)^{\frac12}u,h\big\rangle_{\Omega}\geq 0\,.
\end{equation}
In view of \eqref{iden1/2H1} and \eqref{weaksqrtlap}, we have  
$$\frac{1}{2} \int_{\R^{n+1}_+}|\nabla h^{\rm e}|^2\,\de{\bf x} =\mathcal{E}(h,\Omega)
 =\mathcal{E}(w,\Omega)+\mathcal{E}(u,\Omega) - \big\langle (-\Delta)^{\frac12}u,w\big\rangle_{\Omega} \,,$$
and since 
$$\big\langle (-\Delta)^{\frac12}u,h\big\rangle_{\Omega}= \big\langle (-\Delta)^{\frac12}u,w\big\rangle_{\Omega}-\big\langle (-\Delta)^{\frac12}u,u\big\rangle_{\Omega}= \big\langle (-\Delta)^{\frac12}u,w\big\rangle_{\Omega}-2\mathcal{E}(u,\Omega)\,,$$
inequality \eqref{vac1449} yields 
$$\mathcal{E}(w,\Omega)-\mathcal{E}(u,\Omega) \geq 0\,. $$
Thus $u$ is indeed a minimizing $1/2$-harmonic map in $\Omega$.  
\vskip5pt

\noindent{\it Step 2.}
We now proceed to the construction of the sequence $(h_k)$ satisfying \eqref{approxequmin}.  For an integer $i\geq 1$, we denote by $\chi_i\in C^\infty(\R;[0,1])$ a smooth cut-off function satisfying $\chi_i(t)=1$ for $|t|\leq i$, and $\chi_i(t)=0$ for $|t|\geq i+1$, with $|\chi_i^\prime|\leq C$ for some constant $C$ independent of $i$. We first define 
$$h_i^{(1)}({\bf x}):=\chi_i(x_{n+1})h^{\rm e}({\bf x})\,.$$
By Lemma \ref{estiprelim},  $h_i^{(1)}\in L^2(\R^{n+1}_+)$, so that $h_i^{(1)}\in H^1(\R^{n+1}_+;\R^m)$. Moreover, $h_i^{(1)}=h$ on $\R^n\times\{0\}$, and 
\begin{multline*}
\int_{\R^{n+1}_+}|\nabla h_i^{(1)}|^2\,\de{\bf x}= \int_{\R^{n+1}_+}\chi_i^2|\nabla h^{\rm e}|^2\,\de{\bf  x}\\
+2\int_{\{i<x_{n+1}<i+1\}}\chi_i \chi^\prime_i\, h^{\rm e}\cdot\partial_{n+1} h^{\rm e}\,\de{\bf  x}
+\int_{\{i<x_{n+1}<i+1\}}|\chi^\prime_i|^2|h^{\rm e}|^2\,\de{\bf x}\,.
\end{multline*}
From Lemma \ref{estiprelim} and Fubini's theorem, we infer that 
\begin{equation}\label{coucou1133}
 \int_{\{i<x_{n+1}<i+1\}}|\chi^\prime_i|^2|h^{\rm e}|^2\,\de{\bf x}\leq \frac{C}{i^n}\|h\|^2_{L^1(\R^n)}\,.
 \end{equation}
Since $ h^{\e}\in \dot H^1(\R^{n+1}_+;\R^m)$, it follows by dominated convergence,  \eqref{coucou1133}, and H\"older's inequality, that 
$$\int_{\R^{n+1}_+}|\nabla h_i^{(1)}|^2\,\de{\bf x}\mathop{\longrightarrow}\limits_{i\to\infty}  \int_{\R^{n+1}_+}|\nabla h^{\rm e}|^2\,\de {\bf x}\,.$$
We can thus find an integer $i_k\geq 1$ such that 
\begin{equation}\label{firstapprox1249}
\left|\int_{\R^{n+1}_+}|\nabla h_{i_k}^{(1)}|^2\,\de {\bf x}- \int_{\R^{n+1}_+}|\nabla h^{\rm e}|^2\,\de {\bf x}\right|\leq 2^{-k-2} \,.
\end{equation}
Next we define for an integer $j\geq 1$, 
$$h_j^{(2)}({\bf x}):=\chi_j(|{\bf x}|)h_{i_k}^{(1)}({\bf x})\,.$$
Then $h_j^{(2)}\in H^1(\R^{n+1}_+;\R^m)$, and one classically shows (using $h_{i_k}^{(1)}\in H^1(\R^{n+1}_+;\R^m)$) that
$$\int_{\R^{n+1}_+}|\nabla h_j^{(2)}|^2\,\de {\bf x}\mathop{\longrightarrow}\limits_{j\to\infty}  \int_{\R^{n+1}_+}|\nabla h_{i_k}^{(1)}|^2\,\de {\bf x}\,.$$
In view of \eqref{firstapprox1249}, we can find an integer $j_k\geq 1$ in such a way that 
\begin{equation}\label{firstapprox1249bis}
\left|\int_{\R^{n+1}_+}|\nabla h_{j_k}^{(2)}|^2\,\de {\bf x}- \int_{\R^{n+1}_+}|\nabla h^{\rm e}|^2\,\de {\bf x}\right|\leq 2^{-k-1} \,,
\end{equation}
and $\Omega\times\{0\}\subset \partial^0B_{j_k}$ to ensure that $h_{j_k}^{(2)}=h_{i_k}^{(1)} =h$ on $\R^n\times\{0\}$. 
\vskip3pt

Let us now fix a small parameter $\delta>0$ such that ${\rm dist}(\partial\Omega,\Omega^\prime)>3\delta$, and consider a smooth cut-off function $\psi\in C^{\infty}(\R;[0,1])$ satisfying $\psi(t)=0$ for $|t|<\delta$, and $\psi(t)=1$ for $|t|\geq 2\delta$. For an integer $\ell\geq 1$, we  consider a further cut-off function $\eta_\ell\in C^{\infty}(\R;[0,1])$ such that $\eta_\ell(t)=1$ for $|t|\leq 2^{-\ell}$, $\eta_\ell(t)=0$ for $|t|\geq 2^{-\ell+1}$, and $|\eta_\ell^\prime|\leq C2^\ell$ for some constant $C$ independent of $\ell$. Setting 
$$\zeta_\ell({\bf x}):=1-\eta_\ell(x_{n+1})\psi\big({\rm dist}(x,\Omega^\prime)\big)\,,$$
we define
$$h_\ell^{(3)}({\bf x}):=\zeta_\ell({\bf x})h_{j_k}^{(2)}({\bf x})\,.$$ 
Setting $G_\ell $ to be the interior of the set 
$$\Big(\big\{{\rm dist}(x,\Omega^\prime)\leq 2\delta\,,\; 0\leq x_{n+1}\leq 2^{-\ell} \big\} \cup\big\{x_{n+1}\geq 2^{-\ell}\big\}\Big)\cap B_{j_k}\,,$$
then $G_\ell$ is an admissible bounded  open set satisfying $\overline{\partial^0G_\ell}\subset \Omega\times\{0\}$.  The map $h_\ell^{(3)}$ belongs to $H^1(\R^{n+1}_+;\R^m)$, it is supported in $G_\ell\cup\partial^0G_\ell$, and $h_\ell^{(3)}=h_{j_k}^{(2)}=h$ on the boundary $\R^n\times\{0\}$. 
Then, we have
\begin{multline}\label{vac21244}
\int_{\R^{n+1}_+}|\nabla h_\ell^{(3)}|^2\,\de {\bf x}= \int_{\R^{n+1}_+}\zeta_\ell^2|\nabla h_{j_k}^{(2)}|^2\,\de {\bf  x}\\
+2\int_{\R^{n+1}_+}\zeta_\ell \big(\nabla\zeta_\ell\cdot \nabla h_{j_k}^{(2)}\big)\cdot h_{j_k}^{(2)}\,\de{\bf  x}
+\int_{\R^{n+1}_+}|h_{j_k}^{(2)}|^2|\nabla\zeta_\ell|^2\,\de{\bf x}\,.
\end{multline}
Writing $A_\ell:=\big\{{\rm dist}(x,\Omega^\prime)>\delta\,,2^{-\ell}<x_{n+1}<2^{-\ell+1}\;\big\}$, we estimate
\begin{equation}\label{ven071028}
\int_{\R^{n+1}_+}|h_{j_k}^{(2)}|^2|\nabla\zeta_\ell|^2\,\de{\bf x}\leq C_\delta\int_{A_\ell} \frac{|h_{j_k}^{(2)}|^2}{x^2_{n+1}}\,\de {\bf x} \,.
\end{equation}
Since $h_{j_k}^{(2)}=h=0$ on $\big\{{\rm dist}(x,\Omega^\prime)>\delta\big\}\times\{0\}$, we infer from Hardy's inequality that 
\begin{multline*}
\int_{\big\{{\rm dist}(x,\Omega^\prime)>\delta\big\}\times\R_+} \frac{|h_{j_k}^{(2)}|^2}{x^2_{n+1}}\,\de {\bf x} \leq C\int_{\big\{{\rm dist}(x,\Omega^\prime)>\delta\big\}\times\R_+} |\nabla h_{j_k}^{(2)}|^2\,\de{\bf x} \\ \leq C\int_{\R^{n+1}_+}|\nabla h_{j_k}^{(2)}|^2\,\de{\bf x}\,.
\end{multline*}
As a consequence, 
$$\int_{A_\ell} \frac{|h_{j_k}^{(2)}|^2}{x^2_{n+1}}\,\de{\bf x}\mathop{\longrightarrow}\limits_{\ell\to\infty}0\,,$$
by dominated convergence. In turn, \eqref{ven071028} implies
$$\int_{\R^{n+1}_+}|h_{j_k}^{(2)}|^2|\nabla\zeta_\ell|^2\,\de{\bf x} \mathop{\longrightarrow}\limits_{\ell\to\infty}0\,.$$
Back to \eqref{vac21244}, we deduce (still by dominated convergence and H\"older's inequality) that 
$$\int_{\R^{n+1}_+}|\nabla h_\ell^{(3)}|^2\,\de{\bf x} \mathop{\longrightarrow}\limits_{\ell\to\infty}  \int_{\R^{n+1}_+}|\nabla h_{j_k}^{(2)}|^2\,\de{\bf  x}\,.$$
In view of \eqref{firstapprox1249bis}, we may now select a subsequence $\{\ell_k\}$ such that 
$$\left|\int_{\R^{n+1}_+}|\nabla h_{\ell_k}^{(3)}|^2\,\de{\bf x}- \int_{\R^{n+1}_+}|\nabla h^{\rm e}|^2\,\de {\bf x}\right|\leq 2^{-k} \,,$$
and the conclusion follows for $h_k:=h_{\ell_k}^{(3)}$ and $G_k:=G_{\ell_k}$. 
\end{proof}

As a consequence of Theorem \ref{equivmin}, we can derive a partial regularity  theory for minimizing $1/2$-harmonic from the regularity of minimizing harmonic maps with free boundary (see \cite{MS,Mos}).    
Notice that, in applying Theorem \ref{regfreebd} and Theorem \ref{dimred}, we use that $u^\e\in L^\infty(\R^{n+1}_+)$ by \eqref{Linftyextineq} and the fact that $u$ is taking values in the compact manifold $\mathcal{N}$. 
Recall that ${\rm sing}(u)$ denotes the singular set of $u$ in $\Omega$ (see  \eqref{defsingsetu}), 
which is a relatively closed subset of $\Omega$. 

\begin{corollary}[{\cite{Mos} and \cite[Theorem 1.2 \& Remark 4.24]{MS}}]\label{reghalfharm}
Let $\Omega\subset\R^n$ be a bounded smooth open set. If $u\in \widehat H^{1/2}(\Omega;\mathcal{N})$ is a minimizing $1/2$-harmonic map in $\Omega$, then $u\in C^\infty(\Omega\setminus{\rm sing}(u))$, ${\rm sing}(u)$ is locally finite in~$\Omega$ for $n=2$, and ${\rm dim}_{\mathcal{H}}\,{\rm sing}(u)\leq n-2$ for $n\geq 3$.
\end{corollary}

Exactly as in Theorem \ref{dimred}, the estimate on the Hausdorff dimension of ${\rm sing}(u)$ can be improved according to the existence or non existence of $0$-homogeneous minimizing  $1/2$-harmonic maps, i.e., maps in $H^{1/2}_{\rm loc}(\R^n;\mathcal{N})$ which are minimizing in every ball. 

\begin{definition}\label{def0homodemiharm}
A map $u_0\in H^{1/2}_{\rm loc}(\R^n;\mathcal{N})$ is said to be a $0$-homogenous $1/2$-harmonic map if $u_0$ is $0$-homogeneous and a weakly $1/2$-harmonic map in every ball of $\R^n$. Similarly, $u_0$ is said to be a  $0$-homogenous minimizing $1/2$-harmonic map if it is $0$-homogeneous and a  minimizing  $1/2$-harmonic map in every ball of~$\R^n$.
\end{definition}

\begin{corollary}\label{cordimreduc}
Let $\bar\ell=\bar\ell(\mathcal{N})$ be the largest integer such that any $0$-homogeneous minimizing $1/2$-harmonic map from $\R^{j}$ into $\mathcal{N}$ is a constant for each $j=1,\ldots,\bar\ell$. For any minimizing $1/2$-harmonic map $u$ as in Corollary~\ref{reghalfharm}, we have ${\rm sing}(u)=\emptyset$ if $n\leq \bar\ell$, ${\rm sing}(u)$ is locally finite in $\Omega$ if $n=\bar\ell+1$, and 
${\rm dim}_{\mathcal{H}}\,{\rm sing}(u)\leq n-\bar\ell-1$ if $n\geq \bar\ell+2$. Moreover, $\bar\ell=\ell$ where $\ell$ is given by Theorem~\ref{dimred}. 
\end{corollary}

\begin{proof}
By Theorem  \ref{equivmin}, if $u_0$ is a $0$-homogeneous minimizing $1/2$-harmonic map from $\R^{j}$ into $\mathcal{N}$, then $(u_0)^\e$ is a bounded minimizing harmonic map with free boundary in every half ball $B_R^+$. Since the harmonic extension preserves homogeneity, $(u_0)^\e$ is also $0$-homogeneous. Hence $(u_0)^\e$ is constant whenever $j\leq \ell$ with $\ell$ given by Theorem~\ref{dimred}, and so is $u_0$. This shows that $\ell\leq \bar \ell$. The other way around, if $v_0:\R^{j+1}_+\to \R^d$ with $v_0(\R^j\times\{0\})\subset\mathcal{N}$ is a bounded and $0$-homogeneous  minimizing harmonic map with free boundary, then $v_0=({v_0}_{|\R^j\times\{0\}})^\e$ according to Remark~\ref{rembd0hom}. By Theorem  \ref{equivmin}, it follows that ${v_0}_{|\R^j\times\{0\}}$ is a  $0$-homogeneous minimizing $1/2$-harmonic map from $\R^{j}$ into $\mathcal{N}$. By definition of $\bar \ell$, ${v_0}_{|\R^j\times\{0\}}$ is constant whenever $j\leq \bar \ell$. Hence $v_0$ is constant for $j\leq \bar\ell$, which shows that $\bar\ell\leq \ell$.  We have thus proved that $\bar\ell=\ell$. 
\vskip3pt

Now, if $u$ is as in Corollary~\ref{reghalfharm}, then Theorem  \ref{equivmin} tells us that $u^\e$ is a bounded minimizing harmonic map with free boundary in every admissible bounded open set $G\subset \R^{n+1}_+$ such that $\overline{\partial^0G}\subset\Omega\times\{0\}$. Hence  the conclusion follows from Theorem~\ref{regfreebd} knowing that $\bar\ell=\ell$.  
\end{proof}

%%%%%%%%%%%%%%%%%%%%%%%%%%%%%%%%%%%%%%%%%%%%%%%%%%%%%%%
%%%%%%%%%%%%%%%%%%%%%%%%%%%%%%%%%%%%%%%%%%%%%%%%%%%%%%%
   								       						%%%%%%%%%%%%%%%%%%%%
\section{Minimizing $1/2$-harmonic maps into a sphere} \label{SectionSphere}     %%%%%%%%%%%%%%%%%%%%
								 						%%%%%%%%%%%%%%%%%%%
%%%%%%%%%%%%%%%%%%%%%%%%%%%%%%%%%%%%%%%%%%%%%%%%%%%%%%%
%%%%%%%%%%%%%%%%%%%%%%%%%%%%%%%%%%%%%%%%%%%%%%%%%%%%%%%

\subsection{$1/2$-harmonic circles}\label{sectdefcircles}

The purpose of this first subsection is to recall the notion  $1/2$-harmonic circle into a manifold $\mathcal{N}$, and its relation established in \cite{MS} with $0$-homogeneous $1/2$-harmonic maps from $\R^2$ into $\mathcal{N}$. Once again, $\mathcal{N}$ is assumed to be a smooth and compact submanifold of $\R^m$ without boundary. 
Let us start with the definition of a $1/2$-harmonic circle into $\mathcal{N}$. 
First, the $1/2$-Dirichlet energy of a map $g\in H^{1/2}(\mathbb{S}^1;\mathbb{R}^{m})$ is defined as 
\begin{equation}\label{defencirc}
\mathcal{E}(g,\mathbb{S}^1):=\frac{\gamma_1}{4}\iint_{\mathbb{S}^1\times\mathbb{S}^1}\frac{|g(x)-g(y)|^2}{|x-y|^2}\,\de x\de y\quad\text{with}\quad\gamma_1=\frac{1}{\pi}\,.
\end{equation}
The choice of the constant $\gamma_1$ in \eqref{defencirc} is made in such a way that  
\begin{equation}\label{idenergcircl}
\mathcal{E}(g,\mathbb{S}^1)=\frac{1}{2}\int_{\mathbb{D}}|\nabla w_g|^2\, \de x\qquad \forall g\in H^{1/2}(\mathbb{S}^1;\R^m)\,,
\end{equation}
where $w_g\in H^1(\mathbb{D};\R^m)$ denotes the (unique) harmonic extension of $g$ to the unit disc $\mathbb{D}$ of the plane $\R^2$, i.e., the unique solution of 
\begin{equation}\label{exttodiscwg}
\begin{cases} 
\Delta w_g=0 & \text{in $\mathbb{D}$}\,\\
w_g=g & \text{on $\partial\mathbb{D}=\mathbb{S}^1$}\,,
\end{cases}
\end{equation} 
see e.g. \cite[Section 4.2]{MS}. 

\begin{definition}
A map $g\in H^{1/2}(\mathbb{S}^1;\mathcal{N})$ is said to be a (weakly) $1/2$-harmonic circle into $\mathcal{N}$ if 
$$\left[\frac{\de}{\de t}\mathcal{E}\left(\frac{g+t\varphi}{|g+t\varphi|},\mathbb{S}^1\right)\right]_{t=0}=0 \qquad\forall \varphi\in C^\infty(\mathbb{S}^1;\R^m)\,.$$
\end{definition}
\vskip5pt

\begin{remark}\label{remsmoothharmcirc} 
Any $1/2$-harmonic circle $g$ is smooth, i.e.,  $g\in C^\infty(\mathbb{S}^1)$. This follows directly from the regularity  theory for weakly $1/2$-harmonic maps in one space dimension of \cite{DaRi1,DaRi2} (see also \cite[Theorem~4.18 \&  Remark~4.24]{MS}). Indeed,  as in \cite[Remark~4.29]{MS},  $g\in H^{1/2}(\mathbb{S}^1;\mathcal{N})$ is weakly $1/2$-harmonic if and only if $g\circ \mathfrak{C}_{|\R}\in \dot H^{1/2}(\R;\mathcal{N})$ is a weakly $1/2$-harmonic map on $\R$, where $\mathfrak{C}:\overline{\R^2_+}\to \overline{\mathbb{D}}\setminus\{(1,0)\}$ is the (conformal) Cayley transform (see \eqref{caltrans}) and $ \mathfrak{C}_{|\R}:\R\to\mathbb{S}^1\setminus\{(1,0)\}$ its restriction to $\R\simeq\partial\R^2_+$. Hence the regularity result  of  \cite{DaRi1,DaRi2} applies, and it yields 
 $g\in C^\infty(\mathbb{S}^1\setminus \{(1,0)\})$. On the other hand, the map $\tilde g(x):=g(-x)$ is  clearly $1/2$-harmonic (by invariance of the energy under the symmetry $x\mapsto -x$), so that $\tilde g\in C^\infty(\mathbb{S}^1\setminus \{(1,0)\})$. Thus $g$ is in fact also smooth near $(1,0)$, and the conclusion follows. 
\end{remark}

We are interested in $1/2$-harmonic circles since they appear as angular profiles of $0$-homogeneous $1/2$-harmonic maps on $\R^2$. More precisely, we have the following proposition proved in \cite[Proposition 4.30]{MS}. (Note that this proposition is stated for $\mathcal{N}=\mathbb{S}^1$, but the proof actually applies to any target manifold $\mathcal{N}$.)

\begin{proposition}[\cite{MS}]\label{profhom1/2harm}
A map $u_0\in H^{1/2}_{\rm loc}(\R^2;\mathcal{N})$ is a $0$-homogeneous $1/2$-harmonic map  if and only if $u_0(x)=g\big(\frac{x}{|x|}\big)$ for some $1/2$-harmonic circle $g:\mathbb{S}^1\to \mathcal{N}$. 
\end{proposition}

\begin{remark}
Note that, by Proposition \ref{profhom1/2harm} and Remark \ref{remsmoothharmcirc}, a $0$-homogeneous minimizing $1/2$-harmonic map on $\R^2$ is smooth away from the origin.  
\end{remark}

\subsection{$1/2$-harmonic circles into spheres}\label{seccirclespheres}

The goal of this subsection is to establish a crucial classification result for $1/2$-harmonic circles into spheres, a cornerstone in the proofs of both Theorem \ref{mainthm1} and Theorem \ref{mainthm2}. From now on, we restrict ourselves to the case  $\mathcal{N}=\mathbb{S}^{m-1}$ with $m\geq 3$. 
\vskip3pt

If $g:\mathbb{S}^1\to\mathbb{S}^{m-1}$ is a $1/2$-harmonic circle into $\mathbb{S}^{m-1}$ (and thus smooth), then $w_g$ defines a smooth map from the closed unit disc $\overline{\mathbb{D}}$ into the closed unit ball $\overline{B^m}$ of $\R^m$. By the maximum principle $w_g$ maps the open disc $\mathbb{D}$ into the unit open ball $B^m$, and of course $w_g(\partial \mathbb{D})\subset \mathbb{S}^{m-1}=\partial B^m$ by the boundary condition. In terms of $w_g$, the Euler-Lagrange equation for $g$ being $1/2$-harmonic writes (see e.g. \cite[Remark 4.29]{MS})
\begin{equation}\label{lund1147}
\frac{\partial w_g}{\partial \nu}\wedge w_g=0 \quad\text{on $\partial \mathbb{D}$}\,.
\end{equation}
It has been (independently) proved in \cite{BMRS,Da1,Da2,DaPig}, and \cite[Lemma 4.27 \& Remark 4.29]{MS} that $g$ being $1/2$-harmonic implies that $w_g$ is (weakly) conformal or anti-conformal, i.e., it satisfies 
$$\begin{cases}
\displaystyle \left|\frac{\partial w_g}{\partial x_1}\right|=\left|\frac{\partial w_g}{\partial x_2}\right| \\[10pt]
\displaystyle\frac{\partial w_g}{\partial x_1}\cdot \frac{\partial w_g}{\partial x_2}=0
\end{cases}\quad\text{in $\mathbb{D}$}\,.$$
In addition, $|\nabla w_g|$ does not vanish near $\partial \mathbb{D}$ whenever $g$ is not constant (by the Hopf boundary lemma applied to $|w_g|^2$, see e.g. \cite[Proof of Theorem 2.7]{Da2})\footnote{\noindent One can also prove that $\partial_\nu w_g$ does not vanish on $\partial\mathbb{D}$ as follows. Using \eqref{idenergcircl} and (constrained) outer variations of $\mathcal{E}(\cdot,\mathbb{S}^1)$ at $g$, we can argue as in \cite[Remark 4.3]{MS} to derive the equation
$$\frac{\partial w_g}{\partial\nu}(x)=\left(\frac{\gamma_1}{2}\int_{\mathbb{S}^1} \frac{|g(x)-g(y)|^2}{|x-y|^2}\de y\right)\,g(x) \quad \text{for $x\in\mathbb{S}^1$}\,.$$
Then, assuming by contradiction that $\partial_\nu w_g$ vanishes at some point $x_0\in\mathbb{S}^1$, this equation implies that $g$ is equal to the constant $g(x_0)$ (since $|g|=1$).}.  
As a consequence, if $g$ is not constant, then $w_g$ is a {\sl (branched) minimal immersion} of the unit disc up to the boundary (with branched points only in the interior), and the boundary condition \eqref{lund1147} tells us that $w_g(\overline{\mathbb{D}})$ meets $\partial B^m$ orthogonally. For $m=3$, a celebrated result of J.C.C. Nitsche \cite{Ni} says that $w_g(\overline{\mathbb{D}})$ has to be the intersection of $\overline{B^3}$ with a plane through the origin. This result has been  extended recently to arbitrary dimensions $m\geq 3$ in \cite[Theorem 2.1]{FS}. 
In conclusion, if $g:\mathbb{S}^1\to\mathbb{S}^{m-1}$ is a non constant $1/2$-harmonic circle, then $g(\mathbb{S}^1)$ is an equatorial circle of $\mathbb{S}^{m-1}$. By invariance of the energy under rotations on the image, we can assume that such $1/2$-harmonic map $g$ takes values into $\R^2\times\{0\}^{m-2}\subset  \R^m$,  so that it takes the form $g=(\widehat g,0)$ where $\widehat g:\mathbb{S}^1\to \mathbb{S}^1$ is a non constant $1/2$-harmonic circle.  On the other hand,  the classification of all $1/2$-harmonic circles into $\mathbb{S}^1$ has been obtained in  \cite{BMRS,Da1,Da2,MS}:  they are given by finite Blaschke products (see also \cite{MirPis} for a preliminary result where Blaschke products were first identified). The result can be stated as follows.

\begin{theorem}\label{classifdemicirc}
A map $\widehat g:\mathbb{S}^1\to \mathbb{S}^1$ is a non constant $1/2$-harmonic circle if and only if there exist an integer $d\geq 1$, $\theta\in[0,2\pi]$, and $\alpha_1,\ldots, \alpha_d \in\mathbb{D}$ such that $w_{\widehat g}$ or its complex conjugate equals
$$z\mapsto e^{i\theta}\prod_{k=1}^d \frac{z-\alpha_k}{1-\bar \alpha_k z}\,.$$
In particular, $\mathcal{E}(\widehat g,\mathbb{S}^1)=\pi d$.
\end{theorem}

Gathering the above results, we may now state the following corollary.

\begin{corollary}\label{classifthm}
Assume that $m\geq 3$. If $g:\mathbb{S}^1\to \mathbb{S}^{m-1}$ is a non constant $1/2$-harmonic circle, then $g(\mathbb{S}^1)$ is an equatorial  circle  of $\mathbb{S}^{m-1}$, and $\mathcal{E}(g,\mathbb{S}^1)= \pi d$  with $d=|{\rm deg}(g)|\in\mathbb{N}\setminus\{0\}$. 
\end{corollary}

\subsection{Proof of Theorem \ref{mainthm1}}

We are now ready to prove Theorem~\ref{mainthm1}. According to Corollary~\ref{cordimreduc}, it is enough to prove Proposition~\ref{trivtangmap} below.

\begin{proposition}\label{trivtangmap}
Assume that $m\geq 3$. If $u_0\in H^{1/2}_{\rm loc}(\R^2;\mathbb{S}^{m-1})$ is a $0$-homogeneous minimizing $1/2$-harmonic map, then $u_0$ is constant. 
\end{proposition}

\begin{proof}
Assume by contradiction that $u_0$ is not constant. From Proposition \ref{profhom1/2harm}, we know that 
$$u_0(x)=g\left(\frac{x}{|x|}\right)\,,$$
for some non constant $1/2$-harmonic circle $g:\mathbb{S}^1\to \mathbb{S}^{m-1}$. According to Corollary~\ref{classifthm}, $g(\mathbb{S}^1)$ is an equatorial circle of $\mathbb{S}^{m-1}$, and 
$$ \mathcal{E}(g,\mathbb{S}^1)=\pi d\quad\text{for some integer $d\geq 1$.}$$
Rotating coordinates in the image if necessary, we may assume without loss of generality that $g(\mathbb{S}^1)=\mathbb{S}^1\times \{0\}^{m-2}$. 
\vskip3pt

Let us now fix an arbitrary {\sl radial} function $\zeta\in C^\infty_c(\R^2)$, and define $\varphi(x):=\zeta(x)e_m$, where $(e_1,\ldots, e_m)$ denotes the canonical basis of $\R^m$. 
Then $\varphi\in C^\infty_c(\R^2;\R^m)$, and consider a radius $R=R(\zeta)>0$ such that ${\rm spt}(\zeta)\subset D_R$.  For $\varepsilon\in(-1,1)$, we define 
$$u_\varepsilon:=\frac{u_0+\varepsilon\varphi}{\sqrt{1+\varepsilon^2|\varphi|^2}}\,. $$
Note that $u_\varepsilon\in H_{\rm loc}^{1/2}(\R^2;\R^m)$, and since  $\varphi(x)\cdot u_0(x)=0$ for every $x\not=0$, we actually have $u_\varepsilon\in H_{\rm loc}^{1/2}(\R^2;\mathbb{S}^{m-1})$. By construction we have 
${\rm spt}(u_\varepsilon-u)\subset D_R$,  so that 
$$\mathcal{E}(u_\varepsilon,D_R)\geq \mathcal{E}(u,D_R)$$ 
for every $\eps\in(-1,1)$ by minimality of $u_0$.  Equality obviously holds at $\varepsilon=0$, and thus 
\begin{equation}\label{possecvar}
\left[\frac{\de^2}{\de\varepsilon^2}\mathcal{E}(u_\varepsilon,D_R) \right]_{\eps=0}\geq 0\,. 
\end{equation}
Straightforward computations yield
$$\dot u:=\left({\frac{\de u_\varepsilon}{\de\varepsilon}}\right)_{|\varepsilon=0}= \varphi \quad\text{and} \quad\ddot u:={\frac{\de^2 u_\varepsilon}{\de\varepsilon^2}}_{|\varepsilon=0}= -|\varphi|^2u_0\in H^{1/2}_{00}(D_R;\R^m)\,,$$
and 
\begin{multline*}
\left[\frac{\de^2}{\de\varepsilon^2}\mathcal{E}(u_\varepsilon,D_R) \right]_{\eps=0}=\frac{\gamma_2}{2}\iint_{(\R^2\times\R^2)\setminus(D^c_R\times D^c_R)} \frac{|\dot u(x)-\dot u(y)|^2}{|x-y|^3}\,\de x\de y\\
+\frac{\gamma_2}{2}\iint_{(\R^2\times\R^2)\setminus(D^c_R\times D^c_R)} \frac{(u_0(x)-u_0(y))\cdot(\ddot u(x)-\ddot u(y))}{|x-y|^3}\,\de x\de y\,.
\end{multline*}
Since $|\varphi|^2=\zeta^2$ and $\zeta$ is compactly supported in $D_R$, we obtain 
\begin{equation}\label{jeudec2213}
\left[\frac{\de^2}{\de\varepsilon^2}\mathcal{E}(u_\varepsilon,D_R) \right]_{\eps=0}=\frac{\gamma_2}{2}\iint_{\R^2\times\R^2} \frac{| \zeta(x)-\zeta(y)|^2}{|x-y|^3}\,\de x\de y-\big\langle(-\Delta)^{\frac{1}{2}}u_0,\zeta^2u_0 \big\rangle_{D_R}\,.
\end{equation}
Recalling the weak formulation of \eqref{halfharmmapeqsphere}  (or \cite[Remark 4.3]{MS}), we have
$$ (-\Delta)^{\frac{1}{2}}u_0(x)=\left(\frac{\gamma_2}{2}\int_{\R^2}\frac{|u_0(x)-u_0(y)|^2}{|x-y|^3}\,\de y\right)u_0(x)\quad\text{in $H^{-1/2}(D_R)$}\,.$$ 
Using  the above equation in \eqref{jeudec2213} and the fact that $|u_0|=1$, we deduce from \eqref{possecvar} that 
\begin{equation}\label{preineq}
\iint_{\R^2\times\R^2} \frac{| \zeta(x)-\zeta(y)|^2}{|x-y|^3}\,\de x\de y\geq \int_{\R^2} \left(\int_{\R^2}\frac{|u_0(x)-u_0(y)|^2}{|x-y|^2}\,\de y\right) \zeta^2(x)\,\de x\,. 
\end{equation}
Computing the right hand side of this inequality  in polar coordinates leads to (recall that $\zeta$ is assumed to be radial, i.e., $\zeta(x)=\zeta(|x|)$)
\begin{multline*}
\int_{\R^2} \left(\int_{\R^2}\frac{|u_0(x)-u_0(y)|^2}{|x-y|^2}\,\de y\right) \zeta^2(x)\,\de x\\
=\int_0^\infty \zeta^2(r)\left(\iint_{\mathbb{S}^1\times\mathbb{S}^1}|g(\sigma_1)-g(\sigma_2)|^2\left[\int_0^\infty\frac{\rho}{|\sigma_1-\rho\sigma_2|^3}\,\de\rho\right]\, \de\sigma_1\de\sigma_2\right)\,\de r\,.
\end{multline*}
By formula \cite[GW (213)(5b) p. 326]{Grad}, one has  
\begin{multline*}
\int_0^\infty\frac{\rho}{|\sigma_1-\rho\sigma_2|^3}\,\de\rho=\int_0^\infty\frac{\rho}{(1-2\rho\sigma_1\cdot\sigma_2+\rho^2)^{3/2}}\,\de\rho\\
=\frac{1}{1-\sigma_1\cdot\sigma_2}=\frac{2}{|\sigma_1-\sigma_2|^2} \quad \forall \sigma_1\not=\sigma_2\,.
\end{multline*}
Therefore,  
$$ \int_{\R^2} \left(\int_{\R^2}\frac{|u_0(x)-u_0(y)|^2}{|x-y|^2}\,\de y\right) \zeta^2(x)\,\de x
=\frac{8}{\gamma_1} \mathcal{E}(g,\mathbb{S}^1)\int_0^\infty\zeta^2(r)\,\de r=4\pi d \int_{\R^2}\frac{\zeta^2}{|x|}\,\de x\,,$$
and we conclude from \eqref{preineq} that
\begin{equation}\label{condhardy}
\iint_{\R^2\times\R^2} \frac{| \zeta(x)-\zeta(y)|^2}{|x-y|^3}\,\de x\de y\geq 4\pi d \int_{\R^2}\frac{\zeta^2}{|x|}\,\de x\,.
\end{equation}
In view of the arbitrariness of $\zeta$, we conclude that \eqref{condhardy} holds for every radial function $\zeta\in C^\infty_c(\R^2)$. On the other hand,  Hardy's inequality in $H^{1/2}(\R^2)$ (see e.g. \cite{FSe,Herb}) says that 
\begin{equation}\label{opthardy}
 \iint_{\R^2\times\R^2} \frac{| \zeta(x)-\zeta(y)|^2}{|x-y|^3}\,\de x\de y\geq C_\sharp \int_{\R^2}\frac{\zeta^2}{|x|}\,\de x\quad  \forall \zeta\in C^\infty_c(\R^2)\,,
 \end{equation}
with optimal constant 
$$C_\sharp:=8\pi\left(\frac{\Gamma(3/4)}{\Gamma(1/4)}\right)^2\,.$$
Moreover, the constant $C_\sharp$ is still sharp when restricting \eqref{opthardy} to radial functions (by symmetric decreasing rearrangement, see e.g. \cite{FS}). In view of \eqref{condhardy}, we finally deduce that 
$$4\pi d\leq C_\sharp\,, $$
that is $d\leq 2\left(\frac{\Gamma(3/4)}{\Gamma(1/4)}\right)^2<1$, a contradiction.
\end{proof}

%%%%%%%%%%%%%%%%%%%%%%%%%%%%%%%%%%%%%%%%%%%%%%%%%%%%%%%
%%%%%%%%%%%%%%%%%%%%%%%%%%%%%%%%%%%%%%%%%%%%%%%%%%%%%%%
   								       						%%%%%%%%%%%%%%%%%%%%
\section{Minimizing $1/2$-harmonic maps into the circle} \label{newsection}     %%%%%%%%%%%%%%%%%%%%
								 						%%%%%%%%%%%%%%%%%%%
%%%%%%%%%%%%%%%%%%%%%%%%%%%%%%%%%%%%%%%%%%%%%%%%%%%%%%%
%%%%%%%%%%%%%%%%%%%%%%%%%%%%%%%%%%%%%%%%%%%%%%%%%%%%%%%

The aim of this section is now to prove Theorem \ref{mainthm2} and Theorem \ref{mainthm3}. We thus assume that $n=m=2$. In the first subsection, we recall the construction and properties of the distributional Jacobian in $H^{1/2}$-spaces (see \cite{BBM,Riv} or \cite{MilPi}). In the spirit of~\cite{BCL}, the distributional Jacobian appears to be the main tool to derive energy lower bounds, and in particular to prove the minimality of $\frac{x}{|x|}$, see Section \ref{subsecminimxmodx}. The uniqueness part of Theorem \ref{mainthm2} is proved in Subsection \ref{uniqxmodx}. It relies on Theorem~\ref{classifdemicirc} and  subtle constructions of competitors, again in the spirit of \cite{BCL}. Compared to  \cite{BCL}, the argument is more intricate as it requires a preliminary construction (see Lemma~\ref{lemmodw}) and the numerical evaluation of certain integrals. The last subsection is devoted to the proof of Theorem~\ref{mainthm3}. The proof here is more  classical and it is essentially based on Theorem~\ref{mainthm2}.

\subsection{The distributional Jacobian}

For a map $g\in H^{1/2}(\partial B_1^+;\mathbb{R}^2)$, we define a distribution $T(g)\in ({\rm Lip}(\partial B_1^+))^\prime$ in the following way. Consider $u\in H^1(B_1^+;\mathbb{R}^2)$ such that $u=g$ on $\partial B_1^+$, and set 
$$H(u):=2(\partial_2u\wedge\partial_3 u,\partial_3 u\wedge\partial_1 u,\partial_1u\wedge\partial_2 u) \in L^1(B_1^+;\mathbb{R}^3) \,,$$
where $\wedge$ denotes the wedge product on $\R^2$ (i.e., $a\wedge b:={\rm det}(a,b)$ for $a,b\in\R^2$). 
\vskip3pt

For  a scalar function $\varphi\in {\rm Lip}(\partial B_1^+)$ and an arbitrary extension $\Phi\in {\rm Lip}(B_1^+)$ of $\varphi$ to the closed half ball $\overline{B_1^+}$, we define the action of $T(g)$ on $\varphi$ by setting 
$$\langle T(g),\varphi\rangle:=\int_{B_1^+}H(u)\cdot \nabla\Phi\,\de {\bf x}\,. $$
Noticing that 
$${\rm div}\,H(u)=0\quad\text{in $\mathscr{D}^\prime(B_1^+)$}\,,$$ 
it is routine to check that $T(g)$ is well defined, i.e., it does not depend on the extensions $u$ and $\Phi$, see e.g. \cite[Lemma 3]{BBM}. In addition, the mapping $T:g\mapsto T(g)$ is continuous, see \cite[Lemma 9]{BBM}. 

\begin{lemma}\label{contT}
The mapping $T:H^{1/2}(\partial B_1^+;\mathbb{R}^2)\to ({\rm Lip}(\partial B_1^+))^\prime$ is strongly continuous. More precisely, there exists a constant $C$ such that 
$$\big|\langle T(g_1)-T(g_2),\varphi \rangle\big|\leq C\Big([g_1]_{H^{1/2}(\partial B_1^+)}+ [g_2]_{H^{1/2}(\partial B_1^+)}\Big)[g_1-g_2]_{H^{1/2}(\partial B_1^+)}[\varphi]_{\rm lip}$$
for every $g_1,g_2\in H^{1/2}(\partial B_1^+;\mathbb{R}^2)$ and $\varphi\in {\rm Lip}(\partial B_1^+)$, where 
$$[g]^2_{H^{1/2}(\partial B_1^+)}:=\iint_{\partial B_1^+\times \partial B_1^+} \frac{|g({\bf x})-g({\bf y})|^2}{|{\bf x}-{\bf y}|^3}\,\de\mathcal{H}^2_{\bf x}\de\mathcal{H}^2_{\bf y}$$
and
$$[\varphi]_{\rm lip}:=\sup_{\mathop{{\bf x},{\bf y}\in\partial B_1^+}\limits_{{\bf x}\neq{\bf y}}} \frac{|\varphi({\bf x})-\varphi({\bf y})|}{|{\bf x}-{\bf y}|}\,.$$
\end{lemma}

We shall make use of the following explicit representation of $T(g)$ for maps $g$ belonging to the following class of partially regular maps
\begin{multline*}
{\mathscr R}:=\Big\{g\in H^{1/2}(\partial B_1^+;\mathbb{R}^2) : g_{|\mathbb{D}\times\{0\}}\in W^{1,1}(\mathbb{D};\mathbb{S}^1)\,,\text{ $g$ is smooth on $\overline{\partial^+B_1}$,}\\
\text{smooth in a neighborhood of $\partial \mathbb{D}\times\{0\}$,}\\
\text{and smooth away from finitely many points in $\mathbb{D}\times\{0\}$}\Big\}\,.
\end{multline*}
For a map $g\in \mathscr{R}$ and $a\in\mathbb{D}$ a singular point of $g_{|\mathbb{D}\times\{0\}}: \mathbb{D}\to \mathbb{S}^1$, we shall denote by ${\rm deg}(g,a)$ the topological degree of  $g$ restricted to any small circle around $a$ (oriented in the counterclockwise sense).  We have the following representation of $T(g)$ for $g$ in the class $\mathscr{R}$.

\begin{proposition}\label{calcjac}
Let $g\in{\mathscr R}$ be such that $g\in C^\infty\big((\overline{\mathbb{D}}\times\{0\})\setminus\{a_1,\ldots,a_K\}\big)$ for some distinct points  $a_1,\ldots,a_K\in\mathbb{D}\times\{0\}$. If $d_i:={\rm deg}(g,a_i)$, then 
\begin{equation}\label{repjac}
\langle T(g),\varphi\rangle=2\int_{\partial^+B_1}{\rm det}(\nabla_\tau g)\varphi\,\de\mathcal{H}^2 -2\pi\sum_{i=1}^Kd_i\varphi(a_i)\quad \forall\varphi\in{\rm Lip}(\partial B_1^+)\,,
\end{equation}
where $\nabla_\tau g$ denotes the tangential gradient\footnote{For $x\in\partial^+B_1$ and $\tau_1,\tau_2\in{\rm Tan}(x,\partial^+B)$ such that $(\tau_1,\tau_2,x)$ is a direct orthonormal basis of $\R^3$, we have 
$\nabla_\tau g(x):=(\partial_{\tau_1}g(x),\partial_{\tau_2}g(x))$, and ${\rm det}(\nabla_\tau g(x))$ does not depend on the choice of $\tau_1$ and $\tau_2$.} of $g$ on $\partial^+B_1$. 
\end{proposition}

\begin{proof}
By the smoothness assumption on $g$, we may find an extension $u$ of $g$ which is smooth in $\overline{B_1^+}\setminus \{a_1,\ldots,a_K\}$.  We first claim that 
\begin{equation}\label{prerep}
\langle T(g),\varphi\rangle=2\int_{\partial^+B_1}{\rm det}(\nabla_\tau g)\varphi\,\de\mathcal{H}^2+\int_{\mathbb{D}}(g\wedge\nabla g)\cdot\nabla^\perp\varphi\,\de x-\int_{\partial \mathbb{D}} (g\wedge\partial_\tau g)\varphi\,\de\mathcal{H}^1\,,
\end{equation}
where $\nabla:=(\partial_{x_1},\partial_{x_2})$ and $\nabla^\perp:=(-\partial_{x_1},\partial_{x_2})$ on $\mathbb{D}$, and $\partial_\tau$ denotes the tangential derivation on $\partial\mathbb{D}$ (oriented in the counterclockwise sense).  
Smoothing $u$ near the $a_i's$, we can find a sequence $(u_k)$ of smooth maps over $\overline B_1^+$ such that $u_k=u$ in a neighborhood of $\partial^+B_1$,  $u_k\to u$ strongly in $H^1(B^+)$, and ${u_k}_{|\mathbb{D}\times\{0\}}\to g_{|\mathbb{D}\times\{0\}}$ strongly in $W^{1,1}(\mathbb{D})$ with $\|{u_k}_{|\mathbb{D}\times\{0\}}\|_{L^\infty(\mathbb{D})}\leq 1$. In particular, given an extension $\Phi\in{\rm Lip}(B_1^+)$ of $\varphi$, we have
\begin{equation}\label{approxjac}
 \int_{B_1^+}H(u_k)\cdot \nabla\Phi\,\de{\bf x}\mathop{\longrightarrow}\limits_{k\to\infty} \langle T(g),\varphi\rangle\,.
 \end{equation}
Since ${\rm div}\,H(u_k)=0$, by the divergence theorem we have 
\begin{align}
\nonumber \int_{B_1^+}H(u_k)\cdot \nabla\Phi\,\de{\bf x}&= \int_{\partial^+B_1} {\bf x}\cdot H(u_k)\varphi\,\de\mathcal{H}^2-2\int_{\mathbb{D}}(\partial_1u_k\wedge\partial_2u_k)\varphi\,\de x\\
\label{dim1144} &= 2 \int_{\partial^+B_1}{\rm det}(\nabla_\tau g)\varphi\,\de\mathcal{H}^2-2\int_{\mathbb{D}}(\partial_1u_k\wedge\partial_2u_k)\varphi\,\de x\,.
\end{align}
Noticing that $2\partial_1u_k\wedge\partial_2u_k={\rm curl}(u_k\wedge\nabla u_k)$, a further integration by parts yields
\begin{align}
\nonumber 2\int_{\mathbb{D}}(\partial_1u_k\wedge\partial_2u_k)\varphi\,dx&=-\int_{\mathbb{D}}(u_k\wedge\nabla u_k)\cdot\nabla^\perp\varphi\,\de x+\int_{\partial\mathbb{D}}(u_k\wedge\partial_\tau u_k)\varphi\,\de\mathcal{H}^1\\
\label{dim1146} &= -\int_{\mathbb{D}}(u_k\wedge\nabla u_k)\cdot\nabla^\perp\varphi\,\de x+\int_{\partial\mathbb{D}}(g\wedge\partial_\tau g)\varphi\,\de\mathcal{H}^1\,.
\end{align}
Gathering \eqref{approxjac}-\eqref{dim1144}-\eqref{dim1146} and letting $k\to \infty$ now leads to \eqref{prerep} by dominated convergence. 

To prove \eqref{repjac}, it is now enough to show that 
\begin{equation}\label{dim1156}
\int_{\mathbb{D}}(g\wedge\nabla g)\cdot\nabla^\perp\varphi\,\de x=\int_{\partial \mathbb{D}} (g\wedge\partial_\tau g)\varphi\,\de \mathcal{H}^1-2\pi\sum_{i=1}^K\varphi(a_i)\,.
\end{equation}
To this purpose we consider a sequence $(\varphi_k)$ of Lipschitz functions over $\overline{\mathbb{D}}$ such that $\varphi_k$ is constant in a neighborhood of each $a_i$,  $\varphi_k\to\varphi$ uniformly on $\overline{\mathbb{D}}$, and $\nabla \varphi_k\rightharpoonup\nabla\varphi$ weakly* in $L^\infty(\mathbb{D})$. In this way, 
$$\int_{\mathbb{D}}(g\wedge\nabla g)\cdot\nabla^\perp\varphi_k\,\de x \mathop{\longrightarrow}\limits_{k\to\infty} \int_{\mathbb{D}}(g\wedge\nabla g)\cdot\nabla^\perp\varphi\,\de x \,. $$
Given $k$, we consider $\varepsilon_k>0$ small enough in such a way that $D_{2\varepsilon_k}(a_i)\cap D_{2\varepsilon_k}(a_j)=\emptyset $ for $i\not= j$, $D_{2\varepsilon_k}(a_i)\cap \partial \mathbb{D}=\emptyset$ for each $i$, and $\varphi_k=\varphi_k(a_i)$ in $D_{\varepsilon_k}(a_i)$. Then, 
\begin{align}
\nonumber\int_{\mathbb{D}}(g\wedge\nabla g)\cdot\nabla^\perp\varphi_k\,\de x=&\int_{\mathbb{D}\setminus \bigcup_{i=1}^KD_{\varepsilon_k}(a_i)}(g\wedge\nabla g)\cdot\nabla^\perp\varphi_k\,\de x\\
\nonumber=&-2\int_{\mathbb{D}\setminus \bigcup_{i=1}^KD_{\varepsilon_k}(a_i)}(\partial_1g\wedge\partial_2g)\varphi_k\,\de x+\int_{\partial \mathbb{D}} (g\wedge\partial_\tau g)\varphi_k\,\de\mathcal{H}^1\\
\nonumber&-\sum_{i=1}^K\varphi_k(a_i)\int_{\partial D_{\varepsilon_k}(a_i)}(g\wedge\partial_\tau g)\,\de\mathcal{H}^1\\
\label{dim1158} =& \int_{\partial \mathbb{D}} (g\wedge\partial_\tau g)\varphi_k\,\de\mathcal{H}^1-2\pi \sum_{i=1}^Kd_i \varphi_k(a_i)\,.
\end{align}
In the last identity, we have used the fact that $\partial_1g\wedge\partial_2g={\rm det}(\nabla g)=0$ in the region $\mathbb{D}\setminus \bigcup_{i=1}^KD_{\varepsilon_h}(a_i)$, since $g$ is $\mathbb{S}^1$-valued and smooth in that region. Letting $k\to\infty$ in \eqref{dim1158}  finally leads to \eqref{dim1156}. 
\end{proof}

\subsection{Proof of Theorem \ref{mainthm2}, part 1.}\label{subsecminimxmodx}
 By Theorem \ref{equivmin}, to prove the minimality of $u_\star(x):=\frac{x}{|x|}$, it is enough to prove that its harmonic extension is minimizing, and this is the way we proceed. First, we need to compute explicitly its harmonic extension. 
To this purpose, it is useful to consider the inverse stereographic projection $\mathfrak{S}:\overline{\mathbb{D}}\to \overline{\mathbb{S}^2_+}$ given by 
\begin{equation}\label{invsterproj}
\mathfrak{S}(x):=\left(\frac{2x_1}{1+|x|^2},\frac{2x_2}{1+|x|^2},\frac{1-|x|^2}{1+|x|^2}\right)\,, 
\end{equation}
and its inverse $\mathfrak{S}^{-1}:\overline{\mathbb{S}^2_+}\to \overline{\mathbb{D}}$ (which is the stereographic projection from the south pole): 
\begin{equation}\label{sterproj}
\mathfrak{S}^{-1}({\bf x})=\left(\frac{x_1}{1+x_3}, \frac{x_2}{1+x_3}\right)= \frac{x_1+i x_2}{1+x_3}\,.
\end{equation}
Let us recall that $\mathfrak{S}$ is a conformal transformation.

\begin{lemma}\label{formextxmodx}
The harmonic extension of the map $u_\star(x):=x/|x|$ is given by 
$$ u_\star^{\e}({\bf x})=\frac{x}{|{\bf x}|+x_3}\,.$$
\end{lemma}

\begin{proof}
Since $u_\star$ is $0$-homogeneous, its  harmonic extension $u_\star^\e$ is also $0$-homogeneous. Being harmonic in $\R^{3}_+$, it satisfies 
\begin{equation}\label{eqonS2+}
\begin{cases}
\Delta_{\mathbb{S}^{2}}u_\star^\e=0 & \text{on $\mathbb{S}^{2}_+$}\,,\\[3pt]
u_\star^\e= u_\star & \text{on $\partial \mathbb{S}^{2}_+=\mathbb{S}^1\times\{0\}$}\,,
\end{cases}
\end{equation}
where $\Delta_{\mathbb{S}^{2}}$ denotes the Laplace-Beltrami operator on $\mathbb{S}^2$. Next we define $w:\overline{\mathbb{D}}\to \mathbb{R}^2$ by setting
$$w(x):=u_\star^\e\big(\mathfrak{S}(x)\big)\,, $$
where $\mathfrak{S}$ is the inverse stereographic projection from the closed unit disc into $\overline{\mathbb{S}^{2}_+}$ defined in \eqref{invsterproj}. Since  $\mathfrak{S}$ is conformal, and $\mathfrak{S}(x)=(x,0)$ for $x\in\partial \mathbb{D}$, we infer from \eqref{eqonS2+} that 
$$\begin{cases} 
\Delta w=0 & \text{in $\mathbb{D}$}\,,\\
w(x)=x & \text{on $\partial\mathbb{D}$}\,.
\end{cases}$$
By uniqueness of the harmonic extension, we deduce that $w(x)=x$ for every $x\in \mathbb{D}$, and consequently
$$u_\star^\e({\bf x})= \mathfrak{S}^{-1}({\bf x})=\frac{x}{1+x_3}\quad\text{for every ${\bf x}=(x,x_3)\in \mathbb{S}^2_+$}\,.$$
The conclusion follows by $0$-homogeneity of $u_\star^\e$. 
\end{proof}

In what follows, we keep the notation $u_\star(x):=x/|x|$. In the following lemma, we provide an approximation result to reduce the class of of competitors (to test the minimality of $u_\star$) to the ones belonging to the class $\mathscr{R}$. 

\begin{lemma}\label{smoothappH1/2}
Let $u\in H^{1/2}(\mathbb{D};\mathbb{S}^1)$ be such that $u=u_\star$ in a neighborhood of $\partial\mathbb{D}$. There exists a sequence $(u_k)\subset H^{1/2}(\mathbb{D};\mathbb{S}^1)\cap W^{1,1}(\mathbb{D})$ such that $u_k=u_\star$ in a neighborhood of $\partial\mathbb{D}$,  $u_k$ is smooth away from finitely many points, and $u_k\to u$ strongly in $H^{1/2}(\mathbb{D})$. 
\end{lemma}

\begin{proof}
Identifying $\R^2$ with the complex plane $\mathbb{C}$,  we recall that both $H^{1/2}(\mathbb{D};\mathbb{C})\cap L^\infty(\mathbb{D})$ and $W^{1,1}(\mathbb{D};\mathbb{C})\cap L^\infty(\mathbb{D})$ are Banach algebras. If $\bar u_\star$ denotes the complex conjugate of $u_\star$, the map $w:=\bar u_\star u$ belongs to $H^{1/2}(\mathbb{D};\mathbb{S}^1)\cap W^{1,1}(\mathbb{D})$,  and it is identically equal to one in a neighborhood of $\partial \mathbb{D}$. Extending $w$ by the value one outside $\mathbb{D}$, we can apply the method in \cite[Proof of Theorem 2.16]{MilPi} to produce a sequence $(w_k)\subset H_{\rm loc}^{1/2}(\R^2;\mathbb{S}^1)\cap W^{1,1}_{\rm loc}(\R^2)$ such that $w_k$ is smooth outside a finite subset of $\R^2$, and 
$w_k\to w$ strongly in $H_{\rm loc}^{1/2}(\R^2)$. Using that $w$ equals one near $\partial \mathbb{D}$, a quick inspection of the construction (which is based on a convolution argument with a sequence of mollifiers) shows that $w_k$ is also equal to one near $\partial \mathbb{D}$ (at least for $k$ large enough). Therefore, setting $u_k:=u_\star w_k$, we have 
$u_k\in H^{1/2}(\mathbb{D};\mathbb{S}^1)\cap W^{1,1}(\mathbb{D})$, $u_k$ is equal to $u_\star$ near $\partial \mathbb{D}$, $u_k$ is smooth away from a finite set, and $u_k\to u_\star$ strongly in $H^{1/2}(\mathbb{D})$. 
\end{proof}

We shall need the following theorem which is a slight generalization of \cite[Theorem~7.5]{BCL}. Since the proof follows closely \cite{BCL}  with only minor modifications, we shall omit it.  

\begin{theorem}[\cite{BCL}]\label{thmBCL}
Let $(\mathscr{M},\boldsymbol{\delta})$ be a compact metric space, and $\mu$ a nonnegative Radon measure on $\mathscr{M}$ satisfying $\mu(\mathscr{M})=1$. Given a closed subset $A\subset \mathscr{M}$, $N\geq 1$ distinct points $a_1,\ldots, a_N\subset A$, and $d_1,\ldots,d_N\in\mathbb{Z}$ satisfying $\sum_i d_i=1$, define for $\nu:=\sum_i d_i\delta_{a_i}$,   
$$\mathbf{I}(\nu):= \sup\Big\{\int_{\mathscr{M}}\varphi\,\de\mu-\int_{\mathscr{M}} \varphi\,\de\nu: \varphi\in {\rm Lip}(\mathscr{M})\,,\;[\varphi]_{\rm lip}\leq 1\Big\}\,,$$
with $[\varphi]_{\rm lip}:=\sup_{x\neq y}\frac{|\varphi(x)-\varphi(y)|}{\boldsymbol{\delta(x,y)}}$. Then, 
$${\bf I }(\nu)\geq \min_{c\in A} \int_{\mathscr{M}} \boldsymbol{\delta}(x,c)\,\de\mu_x\,. $$
\end{theorem}

\begin{proof}[Proof of Theorem \ref{mainthm2}: minimality of $u_\star$] 
By Theorem \ref{equivmin}, to prove that $u_\star$ is a $0$-homogeneous minimizing  $1/2$-harmonic map, it is enough to show that $u^\e_\star$ is a minimizing harmonic map with free boundary in every bounded admissible open set $G\subset\R^{3}_+$. In turn, it reduces to prove that   $u^\e_\star$ is a minimizing harmonic map with free boundary in $B_R^+$ for every radius $R>0$. By $0$-homogeneity of $u^\e_\star$, it is enough to show that $u^\e_\star$ is a minimizing harmonic map with free boundary in $B_1^+$. 

First, we compute using Lemma  \ref{formextxmodx}, 
\begin{equation}\label{dim1611}
 {\bf E}(u_\star^\e,B_1^+)=\int_{B_1^+}\frac{\de{\bf x}}{(|{\bf x}|+x_3)^2}=\int_{\partial^+B_1}\frac{\de\mathcal{H}^2}{(1+x_3)^2}=\pi\,.
 \end{equation}
In view of \eqref{dim1611}, it is thus enough to show that 
\begin{equation}\label{condminimxmodx}
{\bf E}(v,B_1^+)\geq \pi
\end{equation} 
for every map $v\in H^1(B_1^+;\mathbb{R}^2)$ such that $v=u_\star^\e$ in a neighborhood of $\partial^+B_1$ and $|v|=1$ on $\mathbb{D}\times\{0\}$.  

Let us consider such a map $v$. From the pointwise inequality $|\nabla v|^2\geq |H(v)|$, we first infer that 
\begin{equation}\label{dim1619}
{\bf E}(v,B_1^+)\geq \frac{1}{2}\int_{B_1^+}|H(v)|\,\de{\bf x} \,.
\end{equation}
Then, consider an arbitrary function $\varphi\in {\rm Lip}(\partial B_1^+)$ satisfying $|\varphi({\bf x})-\varphi({\bf y})|\leq |{\bf x}-{\bf y}|$ for every ${\bf x},{\bf y}\in\partial B_1^+$. By the McShane-Whitney extension theorem, we can find a $1$-Lipschitz function  $\Phi\in {\rm Lip}(B_1^+)$ such that $\Phi_{|\partial B_1^+}=\varphi$.  Since $|\nabla \Phi|\leq 1$ a.e. in $B_1^+$, we deduce from \eqref{dim1619} that 
\begin{equation}\label{lwbdengxmodx}
{\bf E}(v,B_1^+)\geq \frac{1}{2} \int_{B_1^+}H(v)\cdot \nabla\Phi\,\de{\bf x}=\frac{1}{2}\langle T(g),\varphi\rangle\,,
\end{equation}
where $g:=v_{|\partial B_1^+}\in H^{1/2}(\partial B_1^+;\R^2)$ is equal to $u_\star^\e$ in a neighborhood of $\partial^+B_1$. 

By Lemma \ref{smoothappH1/2}, we can find a sequence $(u_k)\subset H^{1/2}(\mathbb{D};\mathbb{S}^1)\cap W^{1,1}(\mathbb{D})$ such that $u_k=u_\star$ in a neighborhood of $\partial \mathbb{D}$, $u_k$ is smooth away from finitely many points in $\mathbb{D}$, and $u_k\to g_{|\mathbb{D}\times\{0\}}$ strongly in $H^{1/2}(\mathbb{D})$. Setting 
$$g_k({\bf x}):=\begin{cases}
g({\bf x}) & \text{if ${\bf x}\in\partial^+B_1$}\,,\\
u_k(x) & \text{if ${\bf x}=(x,0)\in\mathbb{D}\times\{0\}$}\,,
\end{cases}$$
we have $g_k\in {\mathscr R}$, $g_k=u_\star^\e$ in a neighborhood of $\partial^+ B_1$, and $g_k\to g$ strongly in $H^{1/2}(\partial B_1^+)$. 

Let us now fix the index $k$. Since $g_k\in\mathscr{R}$, we can find distinct points $a_1,\ldots,a_{N_k}$ in $\mathbb{D}$ such that $g_k$ is smooth away from the $a_i$'s. In addition, if $d_i:={\rm deg}(g_k,a_i)$, then 
$$\sum_{i=1 }^{N_k}d_i={\rm deg}(g_k,\partial\mathbb{D})={\rm deg}(u_\star,\partial\mathbb{D})=1\,.$$ 
Applying Proposition \ref{formextxmodx} to $g_k$ together with Lemma \ref{formextxmodx} yields
$$\frac{1}{2}\langle T(g_k),\varphi\rangle=\pi\left(\frac{1}{\pi}\int_{\partial^+B_1}\frac{\varphi}{(1+x_3)^2} \,\de\mathcal{H}^2-\sum_{i=1}^{N_k}d_i\varphi(a_i)\right)\,.$$
In turn, applying Theorem \ref{thmBCL} with $\mathscr{M}=\partial B_1^+$ endowed with the Euclidean metric, $A=\overline{\mathbb{D}}\times\{0\}$, $\mu=\frac{1}{\pi}(1+x_3)^{-2}\mathcal{H}^2\res{\partial^+B_1}$, and $\nu=\sum_i d_i\delta_{a_i}$, yields 
\begin{equation}\label{lund1024}
\sup_{[\varphi]_{\rm lip}\leq 1}\frac{1}{2}\langle T(g_k),\varphi\rangle\geq \min_{c\in\overline{\mathbb{D}}\times\{0\}}\int_{\partial^+B_1}\frac{|{\bf x}-c|}{(1+x_3)^2} \,\de\mathcal{H}^2\,.
\end{equation}
Next, observe that the minimum value above is achieved at $c=0$. Indeed, the function
$$V: z\in\overline{\mathbb{D}} \mapsto\int_{\partial^+B_1}\frac{|{\bf x}-(z,0)|}{(1+x_3)^2} \,\de\mathcal{H}^2 $$
is clearly convex, and 
$$\nabla V(0)=-\int_{\partial^+B_1} \frac{x}{(1+x_3)}\,\de\mathcal{H}^2=0\,.$$ 
Going back to \eqref{lund1024}, we have thus proved that 
\begin{equation}\label{cpabordel1517bis}
\sup_{[\varphi]_{\rm lip}\leq 1}\frac{1}{2}\langle T(g_k),\varphi\rangle\geq \int_{\partial^+B_1}\frac{1}{(1+x_3)^2} \,\de\mathcal{H}^2 =\pi\,.
\end{equation}

Now we deduce from Lemma \ref{contT} that 
\begin{equation}\label{cpabordel1517}
\sup_{[\varphi]_{\rm lip}\leq 1}\frac{1}{2}\langle T(g),\varphi\rangle\geq \sup_{[\varphi]_{\rm lip}\leq 1}\frac{1}{2}\langle T(g_k),\varphi\rangle-C[g-g_k]_{H^{1/2}(\partial B_1^+)}\,,
\end{equation}
for a constant $C$ independent of $k$. Gathering \eqref{lwbdengxmodx}, \eqref{cpabordel1517}, and \eqref{cpabordel1517bis}, we obtain 
$${\bf E}(v,B_1^+)\geq \pi -C[g-g_k]_{H^{1/2}(\partial B_1^+)}\,. $$
Letting $k\to\infty$ leads to \eqref{condminimxmodx}, which completes the proof.
\end{proof}

\subsection{Proof of Theorem \ref{mainthm2}, part 2.}\label{uniqxmodx}

The goal of this subsection is to prove that $u_\star(x)=\frac{x}{|x|}$ is the unique $0$-homogeneous $1/2$-harmonic map from $\R^2$ into $\mathbb{S}^1$, up to an orthogonal transformation. This is achieved in two steps. The first one consists in proving that $u_\star$  is the unique $0$-homogeneous $1/2$-harmonic map of degree $\pm 1$ (at the origin), up to an orthogonal transformation (see Proposition \ref{casd=1}). In the second step, we prove that a $0$-homogeneous $1/2$-harmonic map with  a degree (at the origin) different from $\pm 1$  is not minimizing (see Proposition \ref{competitorconcl}). 

\begin{lemma}\label{rep0homharmmap}
If $u_0$ is a nontrivial $0$-homogenous $1/2$-harmonic map from $\R^2$ into $\mathbb{S}^1$, then 
$$u_0^\e({\bf x})=w\circ\mathfrak{S}^{-1}\Big(\frac{{\bf x}}{|{\bf x}|}\Big)\,,$$
where $\mathfrak{S}^{-1}$ is the stereographic projection \eqref{sterproj}, and $w$ is a finite  Blaschke product or the complex conjugate of a finite  Blaschke product. In other words,  
\begin{equation}\label{formblaschkeprod}
w(z)\text{ or }\,\overline w(z)\;= e^{i\theta}\prod_{j=1}^d\frac{z-\alpha_j}{1-\overline{\alpha_j}z} 
\end{equation}
for some $\theta\in[0,2\pi[$, $d\in \mathbb{N}\setminus\{0\}$, and $\alpha_1,\ldots,\alpha_d\in\mathbb{D}$. As a consequence, 
\begin{equation}\label{quantenerg2d}
{\bf E}(u_0^\e,B_1^+)=\pi d\,. 
\end{equation}
\end{lemma}

\begin{proof}
By Proposition \ref{profhom1/2harm}, $u_0(x)=g(\frac{x}{|x|})$ for every $x\not=0$, for some non constant $1/2$-harmonic circle $g:\mathbb{S}^1\to\mathbb{S}^1$. By Theorem \ref{classifdemicirc}, the harmonic extension $w_g$ of $g$ to the unit disc $\mathbb{D}$ (i.e., the solution of \eqref{exttodiscwg}) is of the form \eqref{formblaschkeprod}. Hence, we only have to prove that 
$u^\e_0({\bf x})=w_g\circ\mathfrak{S}^{-1}(\frac{{\bf x}}{|{\bf x}|})$. The argument is exactly as in the proof of Lemma \ref{formextxmodx}. By $0$-homogeneity, $u_0^\e$ solves 
$$\begin{cases} 
\Delta_{\mathbb{S}^2}u_0^\e =0 & \text{on $\mathbb{S}^2_+$}\,,\\[3pt]
u_0^\e({\bf x})= g & \text{on $\partial \mathbb{S}^2_+=\mathbb{S}^1\times\{0\}$}\,.
\end{cases}$$
As a consequence, $u_0^\e\circ\mathfrak{S}$ is harmonic in $\mathbb{D}$, and it equals $g$ on $\partial \mathbb{D}$. In other words,  $u_0^\e\circ\mathfrak{S}=w_g$, and \eqref{formblaschkeprod} follows.  

Next, by $0$-homogeneity of $u^\e_0$, conformal invariance, \eqref{idenergcircl}, and Theorem \ref{classifdemicirc},
$${\bf E}(u^\e_0,B_1^+)=\frac{1}{2}\int_{\partial^+B_1}|\nabla_\tau u_0^\e|^2\,\de\mathcal{H}^2=\frac{1}{2}\int_{\mathbb{D}}|\nabla w_g|^2\,\de z=\mathcal{E}(g,\mathbb{S}^1)=\pi d\,, $$
which completes the proof. 
\end{proof}

\begin{proposition}\label{casd=1}
Let $g:\mathbb{S}^1\to\mathbb{S}^1$ be a $1/2$-harmonic circle such that ${\rm deg}(g)\in\{\pm 1\}$. Assume that $u_0:=g(\frac{x}{|x|})$ is a $0$-homogeneous minimizing $1/2$-harmonic map from $\R^2$ into $\mathbb{S}^1$. Then $g$ is an orthogonal transformation, i.e., $g(x)=Ax$ for some $A\in O(2,\R)$. 
\end{proposition}

\begin{proof}
{\it Step 1.} By Theorem \ref{equivmin}, $u_0^\e$ is a minimizing harmonic map with free boundary in $B_1^+$. Therefore, $u_0^\e$ is stationary in $B_1^+$ in the sense of \cite[Definition 4.10]{MS}, see \cite[Remark~4.13]{MS}. In turn, by \cite[Remark~4.11]{MS} it implies that 
\begin{equation}\label{vend1714}
\int_{B_1^+}\left(|\nabla u_0^\e|^2{\rm div}\, X  -2\sum_{i,j=1}^3(\partial_iu_0^\e\cdot \partial_ju_0^\e)\partial_jX_i\right)\,\de{\bf x}=0
\end{equation}
for every $X:=(X_1,X_2,X_3)\in C^1(\overline B^+_1;\R^3)$ compactly supported in $B_1^+\cup\partial^0 B_1^+$ and such that $X_3=0$ on $\partial^0 B_1^+$. 

We now consider a unit vector $e\in\mathbb{S}^1\times\{0\}$ and an even function $\eta\in C^1(\R)$ compactly supported in $(-1,1)$. Using the vector field $X({\bf x}):=\eta(|{\bf x}|)e$ in \eqref{vend1714}, we obtain 
\begin{equation}\label{vend1716}
\int_{B_1^+}\Big(|\nabla u_0^\e|^2\,{\bf x}\cdot e-2(e\cdot \nabla u_0^\e)\cdot({\bf x}\cdot\nabla u_0^\e)\Big)\eta^\prime(|{\bf x}|)\,\frac{\de{\bf x}}{|{\bf x}|}=0\,. 
\end{equation}
On the other hand, since $u_0$ is $0$-homogeneous, $u_0^\e$ is also $0$-homogeneous. Hence ${\bf x}\cdot\nabla u_0^\e=0$, and by Fubini's theorem, \eqref{vend1716} yields 
$$\left(\int_{\partial^+B_1} |\nabla u_0^\e|^2{\bf x}\cdot e\,\de\mathcal{H}^2\right)\left(\int_0^1\eta^\prime(r)\,\de r\right) =0\,,$$
since  $\nabla u_0^\e$ is homogeneous of degree $-1$. 
By arbitrariness of $\eta$ and $e$, we conclude that 
\begin{equation}\label{balancecond}
\int_{\partial^+B_1} |\nabla u_0^\e|^2 x\,\de\mathcal{H}^2=0
\end{equation}
(recall that ${\bf x}=(x,x_3)$).
\vskip5pt

\noindent{\it Step 2.} Since minimality is preserved under complex conjugation (i.e., $\overline u_0$ is also a $0$-homogeneous minimizing $1/2$-harmonic map), we may assume that ${\rm deg}(g)=1$ (otherwise we consider $\overline g$ instead of $g$). Then we infer from Lemma \ref{rep0homharmmap} that 
$$u_0^\e({\bf x})=w\circ \mathfrak{S}^{-1}\left(\frac{{\bf x}}{|{\bf x}|}\right)\quad\text{with}\quad w(z)=e^{i\theta}\frac{z-\alpha}{1-\bar\alpha z}\,,$$
for some $\theta\in[0,2\pi[$ and $\alpha\in\mathbb{D}$ (where $\mathfrak{S}^{-1}$ is the stereographic projection \eqref{sterproj}). 

By conformal invariance, we have 
\begin{equation}\label{calcbalw}
 \int_{\partial^+B_1} |\nabla u_0^\e|^2 x\,\de\mathcal{H}^2=2\int_{\mathbb{D}}\big|\nabla w(z)\big|^2\frac{z}{1+|z|^2}\,\de z\,.
 \end{equation}
In addition, since $w$ is holomorphic in $\mathbb{D}$, we have 
\begin{equation}\label{calcgradholom}
|\nabla w(z)|^2=2|w^\prime(z)|^2=\frac{(1-|\alpha|^2)}{|1-\overline\alpha z|^4}\,.
\end{equation}
Hence, combining \eqref{balancecond}, \eqref{calcbalw}, and \eqref{calcgradholom} yields
$$\int_{\mathbb{D}}\frac{z}{(1+|z|^2)|1-\overline\alpha z|^4}\,\de z=0 \,,$$
which in turn implies that $\alpha=0$. In other words, $g(z)=e^{i\theta}z$, i.e., $g$ is a rotation.  
\end{proof}

\begin{lemma}\label{lemmodw}
Let $u_0$ be a $0$-homogeneous minimizing $1/2$-harmonic map from $\R^2$ into~$\mathbb{S}^1$. If $u_0^\e({\bf x})=w\circ\mathfrak{S}^{-1}(\frac{{\bf x}}{|{\bf x}|})$ with  $\mathfrak{S}^{-1}$ the stereographic projection \eqref{sterproj}, and
$$w(z)= e^{i\theta}\prod_{j=1}^d\frac{z-\alpha_j}{1-\overline{\alpha_j}z} \,,$$
with $d\in\mathbb{N}\setminus\{0\}$, and $\alpha_1,\ldots,\alpha_d\in\mathbb{D}$, 
then 
$$|w(z)|\leq \left(\frac{3|z|+1}{|z|+3}\right)^d\quad\text{for every $z\in\mathbb{D}$} \,.$$
\end{lemma}

\begin{proof}
The case $d=1$ is a direct consequence of Proposition \ref{casd=1}, so it remains to consider the case $d\geq 2$.  
Set $\boldsymbol{\delta}:=\max_j|\alpha_j|\in [0,1)$. We may assume without loss of generality that $\boldsymbol{\delta}=|\alpha_d|$. Since minimality is preserved under rotations on the image (i.e., $Au_0$ is  a $0$-homogeneous minimizing $1/2$-harmonic map for every $A\in SO(2,\R)$), we can also assume that $\alpha_d\in [0,1)$, so that $\boldsymbol{\delta}=\alpha_d$. Then we write 
$$w(z)=\frac{z-\boldsymbol{\delta}}{1-\boldsymbol{\delta}z}\,\widetilde w(z)\quad\text{with}\quad \widetilde w(z)= e^{i\theta}\prod_{j=1}^{d-1}\frac{z-\alpha_j}{1-\overline{\alpha_j}z}\,.$$
We aim to prove that 
\begin{equation}\label{esticrucdelta}
\boldsymbol{\delta}\leq 1/3 \,,
\end{equation}
which immediately leads to the conclusion since 
$$\frac{|z-\alpha_j|}{|1-\overline{\alpha_j}z|}\leq \frac{|z|+|\alpha_j|}{|\alpha_j||z|+1} \leq\frac{|z|+\boldsymbol{\delta}}{\boldsymbol{\delta}|z|+1} \leq \frac{3|z|+1}{|z|+3}$$
for each $j$ and every $z\in\mathbb{D}$.
\vskip3pt

To prove \eqref{esticrucdelta}, we shall construct suitable competitors to test the minimality of $u_0^\e$ in $B_1^+$ (recall that $u_0^\e$ is a minimizing harmonic map with free boundary in $B_1^+$ by Theorem \ref{equivmin}). 
Given a parameter $\varepsilon\in(0,1)$, we consider a smooth function $\boldsymbol{\beta}:[0,1]\to [0,1]$ such that $\boldsymbol{\beta}(r)=\boldsymbol{\delta}$ in a neighborhood of $r=1$,  $\boldsymbol{\beta}(r)<1$ for $r>\varepsilon$, and  $\boldsymbol{\beta}(r)=1$ for $r\leq \varepsilon$.  Next we consider the smooth map on $\mathbb{D}\times[0,1]$ given by
$$\widehat w(z,r):= \frac{z-\boldsymbol{\beta}(r)}{1-\boldsymbol{\beta}(r)z}\,\widetilde w(z)\,.$$
By construction,  $\widehat w(\cdot, r)$ is a Blaschke product with  $d$ factors for $r>\varepsilon$, and $(d-1)$ factors for $r\leq \varepsilon$ (more precisely, $\widehat w(\cdot, r)=\widetilde w$ for $r\leq\varepsilon$). Setting $g_r:=\widehat w(\cdot, r)_{|\partial\mathbb{D}}$, we then have ${\rm deg}(g_r)=d$ for $r>\varepsilon$, and ${\rm deg}(g_r)=d-1$ for $r\leq\varepsilon$. From \eqref{idenergcircl} and Theorem \ref{classifdemicirc}, we infer that
\begin{equation}\label{mer1538}
\frac{1}{2}\int_{\mathbb{D}}  \big|\nabla_z \widehat w(z, r)\big|^2\,\de z=\mathcal{E}(g_r,\mathbb{S}^1)=\begin{cases}
\pi d & \text{for $r>\varepsilon$}\,,\\
\pi(d-1) & \text{for $r\leq \varepsilon$}\,.
\end{cases}
\end{equation}
In addition, since $|\widetilde w|\leq 1$, we have  the pointwise estimate 
\begin{equation}\label{mer1539}
\left|\frac{\partial\widehat w}{\partial r}(z,r)\right|^2\leq  \frac{|z^2-1|^2}{|1-\boldsymbol{\beta}(r)z|^4}|\boldsymbol{\beta}^\prime(r)|^2=\frac{(1+|z|^2)^2-4z_1^2}{(1-2\boldsymbol{\beta}(r)z+\boldsymbol{\beta}^2(r)|z|^2)^2}|\boldsymbol{\beta}^\prime(r)|^2\,.
\end{equation}

We define a map $v\in H^1(B_1^+;\R^2)$ by setting
$$v({\bf x}):=\widehat w\left(\mathfrak{S}^{-1}\left(\frac{{\bf x}}{|\bf{x}|}\right), |{\bf x}|\right) \quad\text{for ${\bf x}\in B_1^+$}\,.$$
Note that $|v|=1$ on $\partial^0B_1^+$, and that $v=u_0^\e$ in a neighborhood of $\partial^+B_1$. Hence $v$ is an admissible competitor to test the minimality of $u_0^\e$ in $B_1^+$, i.e., 
\begin{equation}\label{mer1545}
{\bf E}(v,B_1^+)\geq {\bf E}(u_0^\e,B_1^+)=\pi d\,,
\end{equation}
where we have used \eqref{quantenerg2d} in the last equality. 

Computing the energy of $v$ in polar coordinates, we obtain  
\begin{multline}\label{cpamer16}
{\bf E}(v,B_1^+)=\int_0^1\left(\frac{1}{2}\int_{\partial^+B_1}\big|\nabla_\tau v(r{\bf x})\big|^2\,\de\mathcal{H}^2\right)\,\de r\\
+\int_\varepsilon^1\left(\frac{r^2}{2}\int_{\partial^+B_1}\big|\partial_r v(r{\bf x})\big|^2\,\de\mathcal{H}^2\right)\,\de r\,. 
\end{multline}
By conformal invariance, we have
\begin{equation}\label{couaill16}
\frac{1}{2}\int_{\partial^+B_1}\big|\nabla_\tau v(r{\bf x})\big|^2\,\de\mathcal{H}^2=\frac{1}{2}\int_{\mathbb{D}} \big|\nabla_z \widehat w(z, r)\big|^2\,\de z\,.
\end{equation}
Combining \eqref{mer1538}, \eqref{cpamer16}, and \eqref{couaill16} yields
\begin{equation}\label{cpamer16bis}
{\bf E}(v,B_1^+)=\pi(d-\varepsilon)
+\int_\varepsilon^1\left(\frac{r^2}{2}\int_{\partial^+B_1}\big|\partial_r v(r{\bf x})\big|^2\,\de\mathcal{H}^2\right)\,\de r\,. 
\end{equation}

Then, recalling that
\begin{equation}\label{pushforstereoproj}
{\mathfrak{S}^{-1}}_\sharp \mathcal{H}^2\res\mathbb{S}^2= \frac{4}{(1+|z|^2)^2}\,\de z\,,
\end{equation}
we obtain 
$$\frac{r^2}{2}\int_{\partial^+B_1}\big|\partial_r v(r{\bf x})\big|^2\,\de\mathcal{H}^2= 2r^2\int_{\mathbb{D}}\frac{\big|\partial_r\widehat w(z,r)\big|^2}{(1+|z|^2)^2}\,\de z\,.$$
In turn, this last identity  together with \eqref{mer1539} and Lemma \ref{integzeroes} yields 
\begin{equation}\label{couaill17}
\frac{r^2}{2}\int_{\partial^+B_1}\big|\partial_r v(r{\bf x})\big|^2\,\de\mathcal{H}^2\leq 2\pi r^2 F\big(\boldsymbol{\beta}^2(r)\big)\big|\boldsymbol{\beta}^\prime(r)\big|^2
\end{equation}
with 
$$F(t):= \left(\frac{t^2-10t+1}{(1+t)^4}\right)\log\left(\frac{(1-t)^2}{4}\right)-\frac{t^2+11t-2}{(1+t)^3} \,. $$
Notice that $F:[0,1)\to\mathbb{R}$ is an increasing function, and that $F(0)=2-2\log(2)>0$. 
\vskip3pt

Gathering  \eqref{mer1545}, \eqref{cpamer16bis},  and \eqref{couaill17} leads to 
\begin{equation}\label{mersoir18}
\pi \varepsilon \leq 2\pi \int_{\varepsilon}^1 r^2 F\big(\boldsymbol{\beta}^2(r)\big)\big|\boldsymbol{\beta}^\prime(r)\big|^2\,\de r\,.
\end{equation}
Next we set $\boldsymbol{\beta}(r)=:\boldsymbol{\gamma}(\varepsilon/r)$, so that $\boldsymbol{\gamma}:[\varepsilon,1]\to[0,1]$ satisfies $\boldsymbol{\gamma}(1)=1$, $\boldsymbol{\gamma}(t)<1$ for $t<1$, and $\boldsymbol{\gamma}(t)=\boldsymbol{\delta}$ in a neighborhood of $t=\varepsilon$. Changing variables in \eqref{mersoir18}, we infer that 
$$1 \leq 2 \int_{\varepsilon}^1 F\big(\boldsymbol{\gamma}^2(t)\big)\big|\boldsymbol{\gamma}^\prime(t)\big|^2\,\de t\,.$$
In view of our arbitrary choice of $\varepsilon$ and $\boldsymbol{\gamma}$, we conclude that 
\begin{equation}\label{mersoir182}
1 \leq 2 \int_{0}^1 F\big(\boldsymbol{\gamma}^2(t)\big)\big|\boldsymbol{\gamma}^\prime(t)\big|^2\,\de t
\end{equation}
for every $C^1$-function $\boldsymbol{\gamma}:[0,1]\to[0,1]$ satisfying $\boldsymbol{\gamma}(0)=\boldsymbol{\delta}$ and $\boldsymbol{\gamma}(1)=1$. Setting
$$G(s):=\int_0^s\sqrt{F(t^2)}\,\de t \,,$$
inequality \eqref{mersoir18} must hold for $\boldsymbol{\gamma}(t):=G^{-1}\big(G(1)t+G(\boldsymbol{\delta})(1-t)\big)$, which returns the inequality $1\leq 2\big(G(1)-G(\boldsymbol{\delta})\big)^2$. Therefore, 
$$1 \leq \sqrt{2} \int_{\boldsymbol{\delta}}^1 \sqrt{F(t^2)}\,\de t=:J(\boldsymbol{\delta})\,.$$
Since $J(1/3)\approx 0.971<1$, we finally reach the conclusion that $\boldsymbol{\delta}\leq 1/3$. 
\end{proof}

\begin{proposition}\label{competitorconcl}
Let $g:\mathbb{S}^1\to\mathbb{S}^1$ be a $1/2$-harmonic circle. If $d:=|{\rm deg}(g)|\geq 2$, then the map $u_0:=g(\frac{x}{|x|})$ is not a $0$-homogeneous minimizing $1/2$-harmonic map from~$\R^2$ into $\mathbb{S}^1$. 
\end{proposition}

\begin{proof}
We argue by contradiction assuming that $u_0$ is a $0$-homogeneous minimizing $1/2$-harmonic map in $\R^2$.  Once again, it implies that $u_0^\e$ is a minimizing harmonic map with free boundary in $B_1^+$ by Theorem \ref{equivmin}. 
By Lemma \ref{rep0homharmmap}, $u_0^\e$ is of the form \eqref{formblaschkeprod}, and without loss of generality we can assume that the map $w$ in \eqref{formblaschkeprod} is equal to the right hand side of \eqref{formblaschkeprod} (otherwise we consider the complex conjugate of $u_0$ instead of $u_0$, which is also minimizing).  

We shall build competitors to test the minimality of $u_0^\e$, and to this purpose we consider the extended complex plane $\mathbb{C}\cup\{\infty\}$.  We also identify $\R^2_+$ with the complex upper half  plane 
$\mathbb{C}_+:=\big\{z\in\mathbb{C}: \mathfrak{Im}(z)> 0\big\}$. We consider the Cayley transform $\mathfrak{C}:\overline{\mathbb{C}_+}\to \overline{\mathbb{D}}\setminus\{1\}$ given by  
\begin{equation}\label{caltrans}
\mathfrak{C}(z):= \frac{z-i}{z+i}\,,
\end{equation}
and its inverse 
\begin{equation}\label{invcaltrans}
\mathfrak{C}^{-1}(z)= \frac{i(1+z)}{1-z}\,.
\end{equation}
Note that $\mathfrak{C}$ maps the real line $\R\times\{0\}=\partial \mathbb{C}_+$ into $\mathbb{S}^1\setminus\{1\}=\partial\mathbb{D}\setminus\{1\}$. In the sequel, we use the (standard) convention 
$$\mathfrak{C}^{-1}(1)=\infty\quad\text{and}\quad \mathfrak{C}(\infty)=1\,. $$

We define a map $f:\overline{\mathbb{D}}\to \overline{\mathbb{C}_+}\cup\{\infty\}$ by setting
\begin{equation}\label{deffcaley}
f(z):=(\mathfrak{C}^{-1}\circ w)(z)\,.
\end{equation}
As a complex valued function, $f$ is a rational function of $z$ with poles (exactly) at the finite set $Z^+_w:=w^{-1}(\{1\})\subset\mathbb{S}^1$. 
In particular,  $f$ is smooth in $\overline{\mathbb{D}}\setminus Z^+_w$.  
In addition, $f(\mathbb{D})=\mathbb{C}_+$, and $f\big(\mathbb{S}^1\setminus Z^+_w\big)=\R\times\{0\}$. 
\vskip3pt

Given a parameter $\varepsilon\in(0,1)$, we consider a smooth function $\boldsymbol{\theta}:[0,1]\to [0,1]$ such that $\boldsymbol{\theta}(r)=1$ in a neighborhood of $r=1$,  $\boldsymbol{\theta}(r)>0$ for $r>\varepsilon$, and  $\boldsymbol{\theta}(r)=0$ for $r\leq \varepsilon$. Next we define the smooth map on $B_1^+$ given by 
$$v({\bf x}):=\mathfrak{C}\left(\frac{1}{\boldsymbol{\theta}(|{\bf x}|)} f\circ\mathfrak{S}^{-1}\left(\frac{{\bf x}}{|{\bf x}|}\right)\right)\,,$$
where $\mathfrak{S}^{-1}$ is the stereographic projection \eqref{sterproj}. With the convention $0/0=\infty$, we observe that $v$ extends smoothly up to $\partial B_1^+$ except for finitely many points in $\partial^0B^+_1$. More precisely, 
setting $Z^-_w:=w^{-1}(\{-1\})\subset\mathbb{S}^1$, the set $Z^-_w$ is finite, and $v$ is smooth in $\overline{B_1^+}\setminus (\varepsilon Z_w^-\times\{0\})$. By construction, $v=1$ in $B_\varepsilon^+$, $|v|=1$ on $\partial^0 B_1^+$, and  $v=u^\e$ in a neighborhood of $\partial^+B_1$.  As our computations will show, $v\in H^1(B_1^+;\R^2)$ so that $v$ is an admissible competitor to test the minimality of $u_0^\e$ in $B_1^+$, i.e., 
\begin{equation}\label{cpaminenergmardsoir}
{\bf E}(v,B_1^+)\geq {\bf E}(u_0^\e,B_1^+)=\pi d\,, 
\end{equation}
where we have used \eqref{quantenerg2d} in the last equality. 

To compute the energy of $v$, it is useful to rewrite $v$ as 
$$v({\bf x})=\widehat w\left(\mathfrak{S}^{-1}\left(\frac{{\bf x}}{|{\bf x}|}\right), |{\bf x}| \right)\,, $$
where  $\widehat w$ is the smooth map defined on $\overline{\mathbb{D}}\times(\varepsilon,1)$ by 
$$\widehat w(z,r) := \mathfrak{C}\left(\frac{1}{\boldsymbol{\theta}(r)}\mathfrak{C}^{-1} \big(w(z)\big)\right)=\mathfrak{C}\left(\frac{1}{\boldsymbol{\theta}(r)}f(z)\right)\,. $$
Notice that for each $r\in (\varepsilon,1)$, $\widehat w(\cdot, r)$ is a Blaschke product with $d$ factors. Indeed, for each $r\in (\varepsilon,1)$, $\widehat w(\cdot, r)$ is clearly holomorphic on $\mathbb{D}$, it is smooth up to $\partial\mathbb{D}$, and 
$|w(\cdot,r)|=1$ on $\partial \mathbb{D}$. By a classical result of Fatou \cite{Fatou}, it implies that $w(\cdot,r)$ is a finite Blaschke product. Since the restriction $g_r$ of $w(\cdot,r)$ to $\partial\mathbb{D}$ is an $\mathbb{S}^1$-valued function of degree $d$, it must be a product of precisely $d$ factors.  Therefore, we can infer from \eqref{idenergcircl} and Theorem \ref{classifdemicirc} that
\begin{equation}\label{cpamar1635}
\frac{1}{2}\int_{\mathbb{D}} \big|\nabla_z\widehat w(z,r)\big|^2\,\de z=\mathcal{E}(g_r,\mathbb{S}^1) =\pi d \qquad\forall r\in(\varepsilon,1)\,.
\end{equation}
On the other hand, a straightforward computation yields for $r\in(\varepsilon,1)$, 
\begin{align}
\nonumber \left|\frac{\partial\widehat w}{\partial r}(z,r)\right|^2&=\left|\mathfrak{C}^\prime\left(\frac{f(z)}{\boldsymbol{\theta}(r)}\right)\right|^2\frac{|f(z)|^2}{\boldsymbol{\theta}^2(r)}|\boldsymbol{\theta}^\prime(r)|^2\\[5pt]
\label{cpamard1644} &=\frac{4|f(z)|^2|\boldsymbol{\theta}^\prime(r)|^2}{\big(\boldsymbol{\theta}^2(r)+2\boldsymbol{\theta}(r)f_2(z)+|f(z)|^2\big)^2}\,,
\end{align}
where $f_2$ denotes the imaginary part of $f$. 

Computing the energy of $v$ in polar coordinates, we obtain
\begin{multline}\label{calcenergpolmardsoir}
{\bf E}(v,B_1^+)=\int_\varepsilon^1\left(\frac{1}{2}\int_{\partial^+B_1}\big|\nabla_\tau v(r{\bf x})\big|^2\,\de\mathcal{H}^2\right)\,\de r\\
+\int_\varepsilon^1\left(\frac{r^2}{2}\int_{\partial^+B_1}\big|\partial_r v(r{\bf x})\big|^2\,\de\mathcal{H}^2\right)\,\de r\,. 
\end{multline}
Using the conformal invariance of $\mathfrak{S}^{-1}$ and \eqref{cpamar1635}, we derive 
\begin{equation}\label{tjcpmardsoir}
\frac{1}{2}\int_{\partial^+B_1}\big|\nabla_\tau v(r{\bf x})\big|^2\,\de\mathcal{H}^2=\frac{1}{2}\int_{\mathbb{D}} \big|\nabla_z \widehat w(z, r)\big|^2\,dz=\pi d\quad\forall r\in(\varepsilon,1)\,.
\end{equation}
Next, \eqref{cpamard1644} together with  \eqref{pushforstereoproj}  leads to 
\begin{align}
\nonumber\frac{r^2}{2}\int_{\partial^+B_1}\big|\partial_r v(r{\bf x})\big|^2\,\de\mathcal{H}^2&= 2\int_{\mathbb{D}} \left|\frac{\partial\widehat w}{\partial r}(z,r)\right|^2 \frac{r^2}{(1+|z|^2)^2}\,\de z\\[3pt]
\label{tjcpmardsoirbis} &=8\int_{\mathbb{D}}\frac{|f(z)|^2|\boldsymbol{\theta}^\prime(r)|^2r^2}{\big(\boldsymbol{\theta}^2(r)+2\boldsymbol{\theta}(r)f_2(z)+|f(z)|^2\big)^2(1+|z|^2)^2}\,\de z\,.
\end{align} 
for every $r\in(\varepsilon,1)$. 

Combining \eqref{cpaminenergmardsoir}, \eqref{calcenergpolmardsoir}, \eqref{tjcpmardsoir}, and \eqref{tjcpmardsoirbis}, we deduce that 
\begin{equation}\label{premcondmarsoir}
\frac{\pi d\varepsilon}{8} \leq \int_\varepsilon^1\left(\int_{\mathbb{D}}\frac{|f(z)|^2|\boldsymbol{\theta}^\prime(r)|^2r^2}{\big(\boldsymbol{\theta}^2(r)+2\boldsymbol{\theta}(r)f_2(z)+|f(z)|^2\big)^2(1+|z|^2)^2}\,\de z\right)  \,\de r \,.
\end{equation}
Next we set $\boldsymbol{\theta}(r)=:\boldsymbol{\alpha}(\varepsilon/r)$, so that $\boldsymbol{\alpha}:[\varepsilon,1]\to [0,1]$ satisfies $\boldsymbol{\alpha}(1)=0$,  $\boldsymbol{\alpha}(t)>0$ for $t<1$, and $\boldsymbol{\alpha}(t)=1$ in a neighborhood of $t=\varepsilon$. Changing variables in \eqref{premcondmarsoir} gives 
\begin{equation}\label{cpamerbalt}
\frac{\pi d}{8} \leq \int_\varepsilon^1H_f\big(\boldsymbol{\alpha}(t)\big) |\boldsymbol{\alpha}^\prime(t)|^2\,\de t
\end{equation}
with 
$$H_f(a):=\int_{\mathbb{D}}\frac{|f(z)|^2}{\big(a^2+2a f_2(z)+|f(z)|^2\big)^2(1+|z|^2)^2}\,\de z\,,\quad a\in(0,1]\,. $$
In view of \eqref{deffcaley}, we can rewrite $H_f(a)$ as 
$$H_f(a)=\int_{\mathbb{D}}\frac{K_a\big(w(z)\big)}{(1+|z|^2)^2}\,\de z\,,$$
where $K_a:\mathbb{D}\to [0,\infty)$ is given by 
$$K_a(z):=\frac{|\mathfrak{C}^{-1}(z)|^2}{(a^2+2a\mathfrak{C}_2^{-1}(z)+|\mathfrak{C}^{-1}(z)|^2)^2} \,,$$
and $\mathfrak{C}_2^{-1}$ denotes the imaginary part of $\mathfrak{C}^{-1}$. 

Since minimality is preserved under rotations on the image,  $\sigma u_0$ is a minimizing $0$-homogeneous $1/2$-harmonic map for each $\sigma\in\mathbb{S}^1$. As a consequence, 
\eqref{cpamerbalt} must hold with $f$ replaced by $f_\sigma:=\mathfrak{C}^{-1}(\sigma w)$ for every $\sigma\in\mathbb{S}^1$. Averaging the resulting inequalities over all $\sigma\in \mathbb{S}^1$ yields
\begin{equation}\label{mardsoirtardcpa2140}
\frac{\pi d}{8} \leq\frac{1}{2\pi} \int_{\mathbb{S}^1}\left(\int_\varepsilon^1H_{f_\sigma}\big(\boldsymbol{\alpha}(t)\big) |\boldsymbol{\alpha}^\prime(t)|^2\,\de t\right) \,\de\sigma=\int_\varepsilon^1 \widetilde H_w\big(\boldsymbol{\alpha}(t)\big) |\boldsymbol{\alpha}^\prime(s)|^2\,\de t
\end{equation}
with 
$$\widetilde H_w(a)=\int_{\mathbb{D}}\frac{\widetilde K_a\big(w(z)\big)}{(1+|z|^2)^2}\,\de z\quad \text{ and }\quad\widetilde K_a(z):=\frac{1}{2\pi}\int_{\mathbb{S}^1}K_a(\sigma z)\,\de\sigma\,.$$
Then observe that $\widetilde K_a(z)$ only depends on $|z|$, i.e.,  $\widetilde K_a(z)= \widetilde K_a(|z|)$. Hence Lemma~\ref{lemmkernel}  tells us that 
$$\widetilde K_a\big(w(z)\big)= \frac{1}{2\pi}\int_{\mathbb{S}^1}K_a\big(|w(z)|\sigma\big)\,\de\sigma=J(a,|w(z)|)\,,$$ 
where the function $\lambda\mapsto J(a,\lambda)$, given by formula \eqref{computefunctJ}, is an increasing function.  
Using that $d\geq 2$, we infer from Lemma \ref{lemmodw} that 
$$|w(z)|\leq \left(\frac{3|z|+1}{|z|+3}\right)^2 \quad\forall z\in\mathbb{D}\,,$$
and as a consequence, 
$$\widetilde H_w(a)\leq 2\pi \int_0^1J\left(a,\frac{(3r+1)^2}{(r+3)^2}\right)\frac{r}{(1+r^2)^2}\,\de r=:2\pi F_1(a)\quad\forall a\in(0,1]\,.$$
Inserting this last inequality in \eqref{mardsoirtardcpa2140} leads to 
$$\frac{d}{16}\leq \int_\varepsilon^1 F_1\big(\boldsymbol{\alpha}(t)\big) |\boldsymbol{\alpha}^\prime(s)|^2\,\de t \,. $$
In view of the arbitrariness of $\varepsilon$ and $\boldsymbol{\alpha}$, we conclude that 
\begin{equation}\label{presqgoodform2141}
\frac{d}{16}\leq \int_0^1 F_1\big(\boldsymbol{\alpha}(t)\big) |\boldsymbol{\alpha}^\prime(s)|^2\,\de t
\end{equation}
for every $C^1$-function $\boldsymbol{\alpha}:[0,1]\to [0,1]$ satisfying $\boldsymbol{\alpha}(0)=1$ and $\boldsymbol{\alpha}(1)=0$. 

Setting 
$$G(\alpha):=\int_0^\alpha\sqrt{F_1(a)}\,\de a $$
inequality \eqref{presqgoodform2141} must hold for $\boldsymbol{\alpha}(t)=G^{-1}\big(G(1)(1-t)\big)$, which returns the inequality $d/16\leq (G(1))^2$. In other words,  
\begin{equation}\label{precpamercrbadmood}
\sqrt{d}\leq 4 \int_0^1\sqrt{F_1(a)}\,\de a \,.
\end{equation}
Now we change variable in this integral setting $t=\frac{1-a}{1+a}$. Using formula \eqref{computefunctJ}, we obtain 
\begin{equation}\label{cpamercrbadmood}
4 \int_0^1\sqrt{F_1(a)}\,\de a =2\int_0^1\sqrt{F_2(t)}\,\de t \leq 2\left(\int_0^1 F_2(t)\,\de t\right)^{1/2}
\end{equation}
with 
\begin{multline*}
F_2(t):=\int_0^1 \bigg(\frac{(2t^2+1)t^2(3r+1)^{12}-(6t^2-1)(3r+1)^8(r+3)^4}{((r+3)^4-(3r+1)^4 t^2)^3}\\[4pt]
+\frac{t^2(3r+1)^4(r+3)^8+(r+3)^{12}}{((r+3)^4-(3r+1)^4 t^2)^3}\bigg)\,\frac{r\,\de r}{(1+r^2)^2}\,.
\end{multline*}
From \eqref{precpamercrbadmood} and \eqref{cpamercrbadmood}, we conclude that $d\leq 4 \int_0^1F_2(t)\,\de t$. However, a direct (numerical) computation provides the estimate $4 \int_0^1F_2(t)\,\de t\simeq 1.93<2$, which contradicts $d\geq 2$, and the proof is complete.  
\end{proof}

\subsection{Proof of Theorem \ref{mainthm3}}\label{subsecthm3}

We complete this section with the proof of Theorem~\ref{mainthm3}, and to this puropse we consider $u\in \widehat H^{1/2}(\Omega;\mathbb{S}^1)$ a minimizing $1/2$-harmonic map in a smooth bounded open set $\Omega\subset \R^2$. By Theorem \ref{reghalfharm}, $u$ is smooth in $\Omega$ away from a locally finite subset of $\Omega$.  Assume that $a\in \Omega$ is a singular point of $u$, and assume without loss of generality that $a=0$. Fix $R>0$ such that $D_{2R}\subset \Omega$ and $u\in C^\infty(D_{2R}\setminus\{0\})$.  Then, 
\begin{equation}\label{degmardavantmang}
d:={\rm deg}(u,0)={\rm deg}(u_{|\partial D_\rho}) \qquad\forall \rho\in(0,2R)\,.
\end{equation}
By Theorem \ref{equivmin}, $u^\e$ is a minimizing harmonic map with free boundary in  $B_R^+$. Therefore, $u^\e$ is stationary in $B_R^+$ in the sense of \cite[Definition 4.10]{MS}, see \cite[Remark~4.13]{MS}. In turn, by \cite[Remark~4.11]{MS} it implies that 
\begin{equation}\label{vend1714end}
\int_{B_R^+}\left(|\nabla u^\e|^2{\rm div}\, X  -2\sum_{i,j=1}^3(\partial_iu^\e\cdot \partial_ju^\e)\partial_jX_i\right)\,\de{\bf x}=0
\end{equation}
for every $X:=(X_1,X_2,X_3)\in C^1(\overline B^+_R;\R^3)$ compactly supported in $B_R^+\cup\partial^0 B_R^+$ and such that $X_3=0$ on $\partial^0 B_R^+$. Arguing as in \cite[Proof of Lemma 5.2, Step 2]{MS}, we infer from \eqref{vend1714end} that 
\begin{equation}\label{monotformfin}
\frac{1}{r}{\bf E}(u^\e,B_r^+) -\frac{1}{s} {\bf E}(u^\e,B_s^+) = \int_s^r\frac{1}{t}\left(\int_{\partial^+B_t} \left|\frac{\partial u^\e}{\partial\nu}\right|^2\,\de\mathcal{H}^2\right)\,\de t\quad\forall \,0<s<r<R\,.
\end{equation}
As a consequence, $r\mapsto \frac{1}{r}{\bf E}(u^\e,B_r^+)$ is non decreasing, and the limit
$$\boldsymbol{\Theta}:=\lim_{r\downarrow 0} \frac{1}{r}{\bf E}(u^\e,B_r^+) $$
exists. Since $0$ is a singular point of $u$ (and thus of $u^\e$), it follows that $\boldsymbol{\Theta}>0$ by e.g. \cite[Theorem 3.4]{HL} (recall our discussion before Theorem \ref{regfreebd}). 
\vskip3pt

We now consider a sequence $\rho_k\downarrow 0$ with $\rho_k\leq R$, and we set for $x\in D_{2R/\rho_k}$, 
$$u_k(x):=u(\rho_k x)\,.$$
Then, $u_k\in \widehat H^{1/2}(D_{2R/\rho_k};\mathbb{S}^1)$,  $u^\e_k({\bf x})=u^\e(\rho_k {\bf x})$, and $u^\e_k\in H^1(B^+_{2R/\rho_k})$ is a minimizing harmonic map with free boundary in $B^+_{2R/\rho_k}$. Since 
\begin{equation}\label{rescenergmar1130}
\frac{1}{r^{n-1}}{\bf E}(u_k^\e,B_r^+)= \frac{1}{(\rho_kr)^{n-1}}{\bf E}(u^\e,B_{\rho_kr}^+) \qquad\forall \,0<r<\frac{R}{\rho_k}\,,
\end{equation}
we infer from \eqref{monotformfin} that ${\bf E}(u_k^\e,B_r^+)$ is bounded with respect to $k$ for every $r<R/\rho_k$. Recalling that $|u_k^\e|\leq 1$ (since $u_k$ is $\mathbb{S}^1$-valued), we can apply Theorem \ref{compactfreebd} to find a (not relabeled) subsequence such that $u_k^\e\to v$ strongly in $H^1(B_r^+)$ for every $r>0$, where $v$ is minimizing harmonic map with free boundary in $B_r^+$ for every $r>0$. Setting $u_0:=v_{|\partial \R^3_+}$, we have $u_k\to u_0$ strongly in $H^{1/2}(D_r)$ for every $r>0$. Hence $u_k^\e\to u_0^\e$ in $L^2(B_r^+)$ for every $r>0$ by \cite[Lemma 2.4]{MS}, which shows that $v=u_0^\e$. In view of \eqref{rescenergmar1130} and the strong convergence of $u_k^\e$, we have 
\begin{equation}\label{constantenergdensmard1218}
\frac{1}{r^{n-1}}{\bf E}(u_0^\e,B_r^+)= \lim_{k\to\infty}\frac{1}{r^{n-1}}{\bf E}(u_k^\e,B_r^+)=\boldsymbol{\Theta}\quad\forall r>0\,.
\end{equation}
In turn, rescaling \eqref{monotformfin} yields
\begin{align*}
 \int_s^r\frac{1}{t}\left(\int_{\partial^+B_t} \left|\frac{\partial u_0^\e}{\partial\nu}\right|^2\,\de\mathcal{H}^2\right)\,\de t&=\lim_{k\to\infty} \int_s^r\frac{1}{t}\left(\int_{\partial^+B_t} \left|\frac{\partial u_k^\e}{\partial\nu}\right|^2\,\de\mathcal{H}^2\right)\,\de t \\
 &= \lim_{k\to\infty} \left(\frac{1}{r^{n-1}}{\bf E}(u_k^\e,B_r^+)- \frac{1}{r^{n-1}}{\bf E}(u_k^\e,B_s^+)\right)\\
 &= 0
 \end{align*}
for every $r>s>0$. Therefore, $u_0^\e$ is $0$-homogeneous, and thus $u_0^\e$ is a $0$-homogeneous minimizing harmonic map with free boundary. Since $\mathbf{\Theta}>0$, we deduce from \eqref{constantenergdensmard1218} that $u_0^\e$ is not constant. Then $u_0$ is a non trivial $0$-homogeneous minimizing $1/2$-harmonic map on $\R^2$ by Theorem \ref{equivmin}. Then Theorem \ref{mainthm2} tells us that $u_0(x)=\frac{Ax}{|x|} $ for some orthogonal matrix $A\in O(2,\R)$. In particular, 
\begin{equation}\label{mardfin1231}
{\rm deg}({u_0}_{|\partial D_r})\in\{\pm 1\}\qquad\forall r>0\,. 
\end{equation}
Now, by the strong $H^1$-convergence of $(u^\e_k)$ and Fubini's theorem, (up to a further subsequence if necessary) we can find $r_*\in(0,1)$ such that $u^\e_k\to  u_0^\e$ strongly in $H^1(\partial^+B_{r_*})$. By continuity of the trace operator,  we have $u_k\to u_0$ strongly in $H^{1/2}(\partial D_{r_*})$. The degree being continuous with respect to the strong $H^{1/2}$-convergence (see \cite{BN}), we deduce from \eqref{mardfin1231} that ${\rm deg}({u_k}_{|\partial D_{r_*}})\in\{\pm 1\}$ for $k$ large enough, that is 
${\rm deg}(u_{|\partial D_{\rho_kr_*}})\in \{\pm 1\} $.  In view of \eqref{degmardavantmang}, we have thus proved that $d\in\{\pm 1\}$, which completes the proof.

%%%%%%%%%%%%%%%%%%%%%%%%%%%%%%%%%%%%%%%%%%%%%%%%%%%%%%%%%%%%%%%%%%%%%%%%%%%%%%%%%%
%%%%%%%%%%%%%%%%%%%%%%%%%%%%%%%%%%%%%%%%%%%%%%%%%%%%%%%%%%%%%%%%%%%%%%%%%%%%%%%%%%
%%%%%%%%%%%%%%%%%%%%%%%%%%%%%%%%%%%%%%%%%%%%%%%%%%%%%%%%%%%%%%%%%%%%%%%%%%%%%%%%%%

\appendix

\section{}

We provide in this appendix some details about the computations performed in Section \ref{uniqxmodx}. 

\begin{lemma}\label{integzeroes}
For every $\boldsymbol{\gamma}\in[0,1)$, 
$$ I(\boldsymbol{\gamma}):=\int_{\mathbb{D}} \frac{(1+|z|^2)^2-4z^2_1}{(1-2\boldsymbol{\gamma}z_1+\boldsymbol{\gamma}^2|z|^2)(1+|z|^2)^2}\,\de z=\pi F(\boldsymbol{\gamma}^2)$$
with 
$$F(t):=\left(\frac{t^2-10t+1}{(1+t)^4}\right)\log\left(\frac{(1-t)^2}{4}\right)-\frac{t^2+11t-2}{(1+t)^3} \,. $$
\end{lemma}

\begin{proof}
Write $I(\boldsymbol{\gamma})=A(\boldsymbol{\gamma})-4B(\boldsymbol{\gamma})$ with 
$$A(\boldsymbol{\gamma}):= \int_{\mathbb{D}} \frac{1}{(1-2\boldsymbol{\gamma}z_1+\boldsymbol{\gamma}^2|z|^2)}\,\de z$$
and
$$B(\boldsymbol{\gamma}):=\int_{\mathbb{D}} \frac{z^2_1}{(1-2\boldsymbol{\gamma}z_1+\boldsymbol{\gamma}^2|z|^2)(1+|z|^2)^2}\,\de z\,.$$
Using polar coordinates, we further rewrite  
$$A(\boldsymbol{\gamma})=\int_0^1 M(\boldsymbol{\gamma}r)r\,\de r\quad\text{and}\quad B(\boldsymbol{\gamma})=\int_0^1\frac{N(\boldsymbol{\gamma}r)r^3}{(1+r^2)^2}\,\de r\,,$$
where
$$M(a):= \int_0^{2\pi} \frac{\de\theta}{(1-2a\cos(\theta)+a^2)^2}\quad\text{and}\quad N(a):=\int_0^{2\pi}\frac{\cos^2(\theta)}{(1-2a\cos(\theta)+a^2)^2}\,\de \theta$$
are defined for $a\in[0,1)$. 

Lengthy but elementary computations yield
$$M(a)=2\pi\frac{1+a^2}{(1-a^2)^3} \quad\text{and}\quad N(a)=2\pi\left(\frac{1+a^2}{(1-a^2)^3}-\frac{1}{2(1-a^2)}\right)\,.$$
Then we first obtain 
\begin{equation}\label{me1051}
A(\boldsymbol{\gamma})=2\pi\int_0^1\frac{\boldsymbol{\gamma}^2r^3+r}{(1-\boldsymbol{\gamma}^2r^2)^3}\,\de r=\pi\left[\frac{r^2}{(1-\boldsymbol{\gamma}^2r^2)^2}\right]^1_0 =\frac{\pi}{(1-\boldsymbol{\gamma}^2)^2}\,.
\end{equation}
Concerning $B(\boldsymbol{\gamma})$, we can rewrite it as 
\begin{equation}\label{me1052}
B(\boldsymbol{\gamma})=\pi\big(2U(\boldsymbol{\gamma}^2)-V(\boldsymbol{\gamma}^2)\big)
\end{equation}
with 
$$U(t):=\int_0^1\frac{(1+tr^2)r^3}{(1-tr^2)^3(1+r^2)^2}\,\de r\quad\text{and}\quad V(t):=\int_0^1\frac{r^3}{(1-tr^2)(1+r^2)^2}\,\de r\,. $$
Once again, elementary computations lead to 
$$V(t)=\frac{1}{2(1+t)^2}\log\left(\frac{2}{1-t}\right)-\frac{1}{4(1+t)}$$
and
$$U(t)=\left(\frac{t^2-4t+1}{2(1+t)^4}\right)\log\left(\frac{2}{1-t}\right) +\frac{1}{8(1-t)^2}+\frac{1}{2}P(t)\,,$$
with 
\begin{multline*}
P(t):=\frac{1}{4(1-t)} +\frac{1}{4(1+t)} -\frac{3}{4(1+t)^2} +\frac{t^2+2t}{(1+t)^2(1-t)}\\
-\frac{t}{(1+t)(1-t)}-\frac{4t^2}{(1+t)^3(1-t)}-\frac{1-t}{2(1+t)^3}\,.
\end{multline*}
Therefore, 
\begin{equation}\label{me1053}
2U(t)-V(t)= \left(\frac{t^2-10t+1}{2(1+t)^4}\right)\log\left(\frac{2}{1-t}\right)+\frac{1}{4(1-t)^2}+P(t)+\frac{1}{4(1+t)}\,.
\end{equation}
A direct computation shows that 
\begin{equation}\label{me1054}
P(t)+\frac{1}{4(1+t)}=\frac{t^2+11t-2}{4(1+t)^3}\,.
\end{equation}
Gathering \eqref{me1051}-\eqref{me1052}-\eqref{me1053}-\eqref{me1054} now leads to $I(\boldsymbol{\gamma})=\pi F(\boldsymbol{\gamma}^2)$ as announced. 
\end{proof}

\begin{lemma}\label{lemmkernel}
Let $\mathfrak{C}$ be the Cayley transform (defined in \eqref{caltrans}). 
For $a\in(0,1]$ and $z\in\overline{\mathbb{D}}$, let
$$K_a(z):= \frac{|\mathfrak{C}^{-1}(z)|^2}{(a^2+2a\mathfrak{C}_2^{-1}(z)+|\mathfrak{C}^{-1}(z)|^2)^2}\,, $$
where $\mathfrak{C}^{-1}_2$ denotes the imaginary part of $\mathfrak{C}^{-1}$.  Define  for $\lambda\in(0,1)$, 
$$J(a,\lambda):=\frac{1}{2\pi}\int_{\mathbb{S}^1}K_a(\lambda \sigma)\,\de\sigma\,.$$
Then, 
\begin{equation}\label{computefunctJ}
J(a,\lambda)=  \frac{(1+t)^4}{16}\left(\frac{(2t^2+1)t^2\lambda^6-(6t^2-1)\lambda^4+t^2\lambda^2+1}{(1-\lambda^2 t^2)^3}\right)\text{ with }t:=\frac{1-a}{1+a}
\end{equation}
for every $\lambda\in[0,1]$. In addition, $\lambda\in[0,1]\mapsto J(a,\lambda)$ is increasing for every $a\in(0,1]$. 
\end{lemma}

\begin{proof}
Recalling that 
$$\mathfrak{C}^{-1}\sharp\mathcal{H}^1\res \mathbb{S}^1=\frac{2}{1+x^2}\,\de x\,,$$
we change variables to obtain 
$$\frac{1}{2\pi}\int_{\mathbb{S}^1}K_a(\lambda z)\,\de\mathcal{H}^1=\frac{1}{\pi}\int_\R\frac{K_a\big(\lambda \mathfrak{C}(x)\big)}{1+x^2}\,\de x\,.$$
Next we set 
$$c:=\frac{1-\lambda}{1+\lambda}\in(0,1)\,,\quad A:=\frac{a+c}{1+ac}\,,\quad B:=c^2+\frac{1}{c^2}\,,$$
to compute 
$$K_a\big(\lambda \mathfrak{C}(x)\big)=\left(\frac{c^2}{(1+ac)^4}\right)\frac{x^4+B x^2+1}{(x^2+A^2)^2}\,.$$
By Lemma \ref{compintAB} below, we have 
$$J(a,\lambda)= \frac{c^2}{(1+ac)^4} \left(\frac{1+A^2}{2A^3}+\frac{B-2}{2A(A+1)^2}\right)\,.$$
In terms of the variables $t$ and $\mu:=\lambda^2\in(0,1)$, we obtain
$$J(a,\lambda)=\frac{(1+t)^4}{16}\left(\frac{(2t^2+1)t^2\mu^3-(6t^2-1)\mu^2+t^2\mu+1}{(1-\mu t^2)^3}\right)\,,$$
which is the announced formula. Next, if
$$f:\mu\in(0,1)\mapsto \frac{(2t^2+1)t^2\mu^3-(6t^2-1)\mu^2+t^2\mu+1}{(1-\mu t^2)^3}\,,$$
we have 
$$f^\prime(\mu)= \frac{4t^2(1-\mu)^2+2\mu(1-t^2)^2}{(1-\mu t^2)^4}>0\,,$$
which shows that $\lambda\mapsto J(a,\lambda)$ is indeed increasing for every $a\in (0,1)$.  
\end{proof}

\begin{remark}
Note that the function $J(a,\lambda)$ defined in \eqref{computefunctJ} can be rewritten as
$$J(a,\lambda)=\frac{(1+t)^4}{32}\left(\frac{(1-\lambda^2)^2}{(1+\lambda t)(1-\lambda t)^3}+\frac{(1-\lambda^2)^2}{(1+\lambda t)^3(1-\lambda t)}+\frac{4\lambda^2}{1-\lambda^2t^2}\right) \,.$$
From this formula, one easily determines the behavior of $J$ as $a\sim 0$ and $\lambda\sim 1$. 
\end{remark}

\begin{lemma}\label{compintAB}
For $A,B>0$, we have 
\begin{equation}\label{intIAB}
\frac{1}{\pi}\int_\R\frac{x^4+Bx^2+1}{(1+x^2)(x^2+A^2)^2}\,dx=\frac{1+A^2}{2A^3}+\frac{B-2}{2A(A+1)^2}\,. 
\end{equation}
\end{lemma}

\begin{proof}
Write $X:=x^2$, and observe that 
\begin{align*}
\frac{X^2+BX+1}{(X+1)(X+A^2)^2}=&\frac{2-B}{(1-A^2)^2}\frac{1}{X+1}\\
&+ \left(1+\frac{B-2}{(1-A^2)^2}\right)\frac{1}{X+A^2}\\
&+\left((1-A^2)+(2-B)\frac{A^2}{1-A^2}\right)\frac{1}{(X+A^2)^2} \,.
\end{align*}
On the other hand,
$$\frac{1}{\pi}\int_\R\frac{dx}{1+x^2}=1\,,\quad\frac{1}{\pi}\int_\R\frac{dx}{x^2+A^2}=\frac{1}{A}\,,\quad\frac{1}{\pi}\int_\R\frac{dx}{(x^2+A^2)^2}=\frac{1}{2A^3}\,, $$
and \eqref{intIAB} follows. 
\end{proof}

\vskip10pt

\noindent{\it Acknowledgements.} V.M. is supported by the Agence Nationale de la Recherche through the project ANR-14-CE25-0009-01 (MAToS).

%=======================
% BIBLIOGRAPHY AND INDEX
%=======================

\end{document}